\documentclass[12pt,amstex,reqno]{amsart}
\usepackage{mathrsfs}
\usepackage{amsfonts}
\usepackage{amsfonts}

\usepackage{epsfig}
\usepackage{amsmath}
\usepackage{amssymb}
\usepackage{amscd}
\usepackage{graphicx}

\usepackage{pstricks}

\input xy
\xyoption{all}

\newtheorem{thm}{Theorem}[section]
\newtheorem{lem}[thm]{Lemma}

\newtheorem{prop}[thm]{Proposition}

\newtheorem{cor}[thm]{Corollary}

\theoremstyle{definition}

\theoremstyle{remark}
\newtheorem{rem}[thm]{Remark}

\DeclareMathOperator{\supp}{supp}

\numberwithin{equation}{section}

\def \N {\mathbb N}

\def \Z {\mathbb Z}

\begin{document}

\title[$\Delta$-weakly mixing subset in positive entropy actions]{$\Delta$-weakly mixing subset in positive entropy actions of a nilpotent group}

\author[K. Liu]{Kairan Liu}
\address{K. Liu: Department of Mathematics, University of Science and Technology of China,
Hefei, Anhui, 230026, P.R. China}
\email{lkr111@mail.ustc.edu.cn}

\keywords{$\Delta$-transitivity, $\Delta$-weakly mixing, topological entropy, nilpotent groups}


\begin{abstract}
The notion of $\Delta$-weakly mixing subsets is introduced for countable torsion-free discrete group actions.
It is shown that for a finitely generated torsion-free discrete nilpotent group action, positive topological entropy implies the existence of $\Delta$-weakly mixing subsets, and while there exists a  finitely generated torsion-free discrete solvable group action which has positive topological entropy but without any $\Delta$-weakly mixing subsets.
\end{abstract}

\maketitle

\section{Introduction}
In this paper, let $\mathcal{T}$ be a countable discrete group with the unit $\theta_{\mathcal{T}}$.
By a \emph{$\mathcal{T}$-system} $(X,\mathcal{T})$ we mean a compact metric space $X$ endowed with a metric $\rho$, together with $\mathcal{T}$ acting on $X$ by homeomorphism, that is,
there exists a continuous map $\Psi: \mathcal{T}\times X\to X$, $\Psi(T,x)=Tx$ satisfying $\Psi(\theta_{\mathcal{T}},x)=x$, $\Psi(T,\Psi(S,x))=\Psi(TS,x)$ for each $T,S\in\mathcal{T}$ and
$x\in X$.
When $\mathcal{T}$ is the group $\mathbb{Z}$ of integers,
it is generated by the element $1$, in this case we let
$T\colon X\to X$, $x\mapsto \Psi(1,x)$ and denote this dynamical system by $(X,T)$. For a $\mathcal{T}$-system $(X,\mathcal{T})$ and $m\in\N$, $(X^m,\mathcal{T})$ is also a $\mathcal{T}$-system, where $X^m:=X\times X\times\dotsc \times X$ ($m$-times), and $T(x_1,x_2,\dotsc ,x_m):=(Tx_1,Tx_2,\dotsc ,Tx_m)$ for any $(x_1,x_2,\dotsc ,x_m)\in X^m$ and $T\in\mathcal{T}$.

Recurrence is one of the central topics in the study of $\mathcal{T}$-systems.
In 1978, Furstenberg and Weiss published a topological theorem generalizing Brikhoff's recurrence theorem and having interesting combinational corollaries, where $\mathcal{T}$ is an abelian group (see \cite{FW}). As a simple example due to Furstenberg shows that the statement is not valid when the assumption that $\mathcal{T}$ is commutative is omitted (see \cite[P. 40]{F}). In \cite{Leibman0}, Leibman proved the following conjecture, due to Yuzvinsky, formulated by Hendrick$:$ the multiple recurrence theorem holds true when $\mathcal{T}$ is nilpotent. In \cite{Huang-Shao-Ye}, Huang, Shao and Ye obtained a topological analogue of multiple ergodic averages of weakly mixing minimal systems for nilpotent group actions.

Recently, for $\Z$-systems many researchers studied strong forms of multiple recurrence. They introduced and investigated the $\Delta$-transitivity and $\Delta$-weakly mixing (see \cite{Glasner,B_H,Chen-Li-Lv,Chen-Li-Lv2,Kwietniak-Li-Oprocah-Ye,Kwietniak-Oprocah,Moothathu}). A $\Z$-system $(X,T)$ is \emph{$\Delta$-transitive} if for every integer $d\geq 2$ there exists a residue subset $X_0$ of $X$ such that for each $x\in X_0$, $\{(T^nx, T^{2n}x,\dotsc ,T^{dn}x):n\in\N\}$ is dense in the $d$-th product metric space $X^d$. Glasner showed that for a minimal system, weak mixing implies $\Delta$-transitivity (see \cite{Glasner}). In \cite{Moothathu}, Moothathu proved $\Delta$-transitivity implies weak mixing, but there exist strongly mixing systems which are not $\Delta$-transitive.
A $\Z$-system is called \emph{$\Delta$-weakly mixing} if $(X^m,T^{(m)})$ is $\Delta$-transitive for any $m\in\N$.
In \cite{H-L-Y-Z} Huang, Li, Ye and Zhou studied $\Delta$-transitivity and $\Delta$-weak mixing
and proved that for a $\Z$-system $\Delta$-weakly mixing
is in fact equivalent to $\Delta$-transitivity (see \cite[Proposition 3.2]{H-L-Y-Z}) but it is no longer true for $\Delta$-weakly mixing sets and $\Delta$-transitive sets (see \cite[Remark 3.5]{H-L-Y-Z}).

Inspired by the above ideas and results, we introduce  $\Delta$-transitivity and $\Delta$-weak mixing for a countable torsion-free discrete group $\mathcal{T}$-action.
Recall that a group is called \emph{torsion-free}
if any element has infinite order except the identity element.
Let $(X,\mathcal{T})$ be a $\mathcal{T}$-system,
where $\mathcal{T}$ is a countable torsion-free discrete group,
and $E$ be a closed subset of $X$ with $|E|\geq 2$.
We say that $E$ is a \emph{$\Delta$-transitive subset} of $(X,\mathcal{T})$ provided that there is a residue subset $A$ of $E$ such that for any $x\in A$, $d\geq 1$ and pairwise distinct $T_1,T_2,\dotsc,T_d\in \mathcal{T}\setminus\{\theta_{\mathcal{T}}\}$, the orbit closure of the $d$-tuple $(x,x,\dotsc,x)$ under the action $T_1\times T_2\times\dotsb\times T_d$ contains $E^d$, i.e.
$$\overline {orb_{+}((x,x,\dotsc,x), T_1\times T_2\times\dotsb\times T_d)}\supseteq E^d. $$
where $orb_{+}((x,x,\dotsc,x),T_1\times T_2\times\dotsb\times T_d):=\{(T_1^nx,T_2^nx,\dotsc,T_d^n x)\colon  n\in\N\}$, and
$E^d:=E\times E\times\dotsb\times E$ ($d$ times),
and a \emph{$\Delta$-weakly mixing subset of $(X,\mathcal{T})$} if $E^m$ is a $\Delta$-transitive subset of $(X^m,\mathcal{T})$ for any $m\in\mathbb{N}$.
If $X$ is a $\Delta$-transitive (reps.\ $\Delta$-weakly mixing) subset of $(X,\mathcal{T})$
then we say that the $\mathcal{T}$-system $(X,\mathcal{T})$ is $\Delta$-transitive (reps.\ $\Delta$-weakly mixing). We will show that if $(X,\mathcal{T})$ is $\Delta$-weakly mixing then $X$ is perfect and $(X,\mathcal{T})$ is weakly mixing (see Proposition \ref{weakly mixng}).

It is well known that for a $\Z$-system $(X,T)$, there always  exists a $T$-invariant Borel probability measure on $X$,
whereas for some groups $\mathcal{T}$ there do not exist any invariant Borel probability measures on  a $\mathcal{T}$-system, while the amenability of the acting group $\mathcal{T}$ ensures the existence of invariant Borel probability measures.

For $\Z$-systems,  the variational principle of entropy and the Shannon-McMillan-Breiman (SMB) theorem are important in the study the
entropy theory (see \cite{Parry,Goodman,Goodwyn,Dinaburg}). Comparing to $\Z$-systems,  the study of dynamical systems with amenable group actions lagged behind, while the situation is rapidly changed in recent years, many researchers studied the entropy theory of dynamics with amenable group actions (see \cite{Emerson,Kieffer,Connes-Feldman-Weiss,Ornstein-Weiss,Ward-zhang,Rudolph,Glasner-Thouvenot-Weiss,Danilenko,Lindenstrauss}).
Lindenstrauss and Weiss obtained a generalization of the SMB theorem (see \cite{Lindenstrauss,Weiss}). The variational principle of
entropy also holds for countable discrete amenable groups (see \cite{Ollagnier-Pinchon,Stepin-Tagi-Zade}).
In \cite{H-Y-Z}, Huang, Ye and Zhang established a local variational principle  for the entropy of a given finite open cover of a countable discrete amenable group actions.
Glasner, Thouvenot and Weiss proved an important disjointness theorem which asserts the relative disjointness in the sense of Furstenberg, of zero-entropy extension from completely positive entropy extensions.
An application of this theorem is to obtain that the Pinsker factor of a product system is equal to the product of the Pinsker factor of the component systems (see \cite{Glasner-Thouvenot-Weiss}).
In \cite{Kerr-Li} Kerr and Li showed that for a amenable group action positive entropy indeed implies Li-Yorke chaos.

In \cite{B_H},
Blanchard and Huang introduced a local version of weakly mixing for $\Z$-systems and showed that positive entropy implies the existence
of weakly mixing sets.
Recently, Huang, Li, Ye and Zhou proved that for a $\Z$-system $(X,T)$ positive entropy indeed implies the existence of $\Delta$-weakly mixing sets (see \cite[Theorem C]{H-L-Y-Z}).
In this paper, we generalize this result for dynamical systems of finitely generated torsion-free nilpotent group actions.

\begin{thm}\label{thm-A} Let $(X,\mathcal{T})$ be a $\mathcal{T}$-system, where $\mathcal{T}$ is a finitely generated torsion-free discrete nilpotent group. If $h_{top}(X,\mathcal{T})>0 $ then there exist $\Delta$-weakly mixing subsets of $(X,\mathcal{T})$.
\end{thm}

It should be noticed that after modifying the example introduced by Furstenberg in \cite[P.40]{F} we construct a dynamical system $(Z,F)$ with positive topological entropy, where $Z$ is a compact metric space and $F$ is a finitely generated torsion-free discrete solvable group, such that there is no $\Delta$-transitive subsets for $(Z,F)$
(see Proposition \ref{solvable2}).

Following the idea of Theorem A in \cite{H-L-Y-Z},
we also give an equivalent characterization for $\Delta$-weakly mixing subsets as follows.
\begin{thm}\label{thm-B} Let $(X,\mathcal{T})$ be a $\mathcal{T}$-system, where $\mathcal{T}$ is a countable torsion-free discrete group. If $E$ is a closed subset of $X$ with $|E|\geqslant 2$, then $E$ is a $\Delta$- weakly mixing subset of $(X,\mathcal{T})$ if and only if $E$ is perfect and there exists an increasing sequence of Cantor sets $C_1\subset C_2\subset\dotsc $ of $E$ such that $C=\bigcup_{i=1}^{\infty}C_i$ is dense in $E$, and has the following property: for any subset $A$ of $C$, $d\in \mathbb{N}$, pairwise distinct $T_1,T_2,..,T_d\in \mathcal{T}\setminus \{ \theta_{\mathcal{T}}\}$ and any continuous functions $g_j: A\to E$ for $j=1,2,\dotsc ,d$ there exists an increasing sequence $\{q_k\}_{k=1}^{\infty}$ of positive integers such that
$$\lim_{k\rightarrow \infty }T_j^{q_k}x=g_j(x)$$
for every $x\in A$ and $j=1,2,\dotsc ,d$.
\end{thm}

The term "chaos" in connection with a map was introduced by Li and York and proved its value for interval maps (see \cite{Li-Yorke}). In \cite{B-G-K-M}, the authors considered the Li-Yorke definition of chaos in the setting of general topological dynamics $(X,T)$ and proved that positive entropy implies Li-Yorke chaos. In sofic context, positive topological entropy with respect to some sofic approximation sequence implies Li-Yorke chaos (see \cite{Kerr-Li2}). Comparing  with the Li-York chaos we have the following definition of chaos.

We say that a $\mathcal{T}$-dynamical system $(X,\mathcal{T})$ is \emph{asynchronous chaotic} if there exists an increasing sequence of Cantor sets $C_1\subset C_2\subset\dotsc $ of $X$ and $\delta>0$  such that for any distinct $x,y\in C:=\bigcup_{i=1}^{\infty}C_i$  and  $T_1,T_2\in\mathcal{T}\setminus\{\theta_{\mathcal{T}}\}$,
$$\liminf_{n\to\infty}\rho(T_1^nx,T_2^ny)=0, \ \limsup_{n\to\infty}\rho(T_1^nx,T_2^ny)\geq\delta,$$
such $C$ is called a \emph{asynchronous chaotic set}.

Following Theorem \ref{thm-A} and Theorem \ref{thm-B}, we have the following corollary.
\begin{cor}\label{corollary} Let $(X,\mathcal{T})$ be a $\mathcal{T}$-system, where $\mathcal{T}$ is a finitely generated torsion-free discrete nilpotent group. If $h_{top}(X,\mathcal{T})>0 $, then $(X,\mathcal{T})$ is asynchronous chaotic.
\end{cor}

For a countable torsion-free discrete amenable group actions, we do not know whether Corollary \ref{corollary} still holds. More precisely, assume that $\mathcal{T}$ be a countable torsion-free discrete amenable group and $(X,\mathcal{T})$ is a $\mathcal{T}$-system with $h_{top}(X,\mathcal{T})>0$, is $(X,\mathcal{T})$ asynchronous chaotic?

This paper is organized as follows. In Section 2, we recall some basic concepts and useful properties. In Section 3, we will present some properties of $\Delta$-weakly mixing subsets of a $\mathcal{T}$-system. In Section 4, we will present some concepts and properties of nilpotent group, then give the proof of Theorem \ref{thm-A}. Finally, we will prove Theorem \ref{thm-B} and Corollary 1.3 in Section 5.

\section{Preliminaries}
In this section, we will review the hyperspace $2^X$ of a compact metric space $X$ with the Hausdorff metric, density of subsets of non-negative integers, extension, entropy of an amenable group action. We also present some basic results which will used later.
\subsection{Hyperspace of space}
For a compact metric space $X$ with a metric $\rho$, the Hausdorff metric of two non-empty compact subsets $A,B$ of $X$ is defined as:
$$\rho_{H}(A,B)=\max\{\max_{x\in A}\min_{y\in B}\rho(x,y), \ \max_{y\in B}\min_{x\in A}\rho(x,y) \}.$$
The metric space $(2^X,\rho_H)$ (\emph{hyperspace of $X$}) is compact since $(X,\rho)$ is compact, where $2^X$ is the collection of all non-empty compact subsets of $X$.
For non-empty open subsets $U_1,U_2,\dotsc ,U_n$ of $X$, let
$$ \langle U_1,U_2,\dotsc ,U_n\rangle :=
\biggl\{A\in 2^X \colon A\subset \bigcup_{i=1}^{n}U_i \text{ and  } A \cap U_i \neq \emptyset,\  i= 1,2,\dotsc ,n \biggr\}. $$
Collection of those $ \langle U_1,U_2,\dotsc ,U_n \rangle $ form a basis for the Hausdorff topology of $2^X$ induced by $\rho_{H}$, where $U_1,U_2,\dotsc ,U_n$ are non-empty open subsets of $X$.

A subset $Q$ of $2^X$ is called \emph{hereditary} if $2^A\subset Q$ for every set $A\in Q$. For a hereditary subset of $2^X$, there is a consequence of the Kuratowski-Mycielski Theorem (\cite[Theorem 5.10]{Akin}).
\begin{lem}\label{K_M} Suppose that $X$ is a perfect compact metric space. If a hereditary subset $Q$ of $2^X$ is residual, then there exists an increasing sequence of Cantor subsets $C_1 \subset C_2\subset \dotsb $ of $X$  such that $C_i\in Q$ for every $i\geq 1$ and $C=\bigcup_{i=1}^\infty C_i$ is dense in $X$.
\end{lem}

\subsection{Density of subsets of non-negative integers}
Let $\Z$, $\Z_+$ and $\N$ denote the collection of
all integers, non-negative integers and positive integers respectively.
The \emph{lower density} and \emph{upper density} of a subset $F\subseteq \Z+$ is defined respectively by
$$\underline D(F)=\liminf_{n\to\infty}\frac{|F\cap \{0,1,\dotsc ,n-1 \}|}{n}$$
and
$$\overline D(F)=\limsup_{n\to\infty}\frac{|F\cap \{0,1,..,n-1\}|}{n}. $$
We say that $F$ has density $D(F)$ if $\underline D(F)=\overline D(F)$, where $D(F)$ denote this common value.
There is a simple fact that we will use in the Section 4: for a real sequence $\{a_n\} _{n=0}^{\infty}$ with $0\leq a_n \leq M$ for some positive real number $M$ if
$$\liminf_{N\to \infty}\frac{1}{N}\sum_{i=0}^{N-1}a_n>0$$
then $\underline D(E)>0$ where $E:=\{n\in \Z_+\colon a_n>0\}$.

\subsection{Entropy of an amenable group action}
A countable discrete group $\mathcal{T}$ is called \emph{amenable} if there exists a sequence of non-empty finite subsets
$\{F_n\}_{n=1}^{+\infty}$ of $\mathcal{T}$ such that
$$\lim_{n\to +\infty}\frac{|TF_n\Delta F_n|}{|F_n|}=0$$
holds for every $T\in \mathcal{T}$,
and such $\{F_n\}_{n=1}^{+\infty} $ is called a \emph{F\o{}lner  sequence} of $\mathcal{T}$.
We know that finite groups, solvable groups and finitely generated groups of subexponential growth are all amenable groups.

Let $(X,\mathcal{T})$ be a $\mathcal{T}$-system, where $\mathcal{T}$ is a countable discrete amenable group with F\o{}lner sequence $\{F_n\}_{n=1}^{+\infty}$. A finite cover of $X$ is a finite family of Borel subsets of $X$, whose union is $X$. Denote by $\mathcal{C}_X^o$ the collection of all open finite covers of $X$. Let $\mathcal{U}\in \mathcal{C}_X^o$.
We set $N(\mathcal{U})$ to be the minimum among the cardinalities of all sub-families of $\mathcal{U}$ covering $X$ and define $H(\mathcal{U})=\log N(\mathcal{U})$. For a finite non-empty subset $\mathcal{S}$ of $\mathcal{T}$ and $\mathcal{U}\in \mathcal{C}_X^o$, set
\[\mathcal{U}_\mathcal{S}=
\bigvee_{T\in \mathcal{S}}T^{-1}\mathcal{U}=
\biggl\{\bigcap_{T\in \mathcal{S}}T^{-1}U_T \mid  U_T\in \mathcal{U}\biggr\}.\]
The \emph{topological entropy of $\mathcal{U}$},
$$h_{top}(\mathcal{T},\mathcal{U})=\lim_{n\to\infty}\frac{1}{|F_n|}H(\mathcal{U}_{F_n})$$
exists and is independent of the F\o{}lner sequence (see \cite[Theorem 6.1]{Lindenstrauss-Weiss}).
The \emph{topological entropy  of $(X,\mathcal{T})$} is defined by
$$h_{top}(X,\mathcal{T}):=\sup_{\mathcal{U}\in \mathcal{C}_X^o}h_{top}(\mathcal{T},\mathcal{U}). $$

Denote by $\mathcal{M}(X)$ the set of all Borel probability measures on $X$. $\mu\in\mathcal{M}(X)$ is called \emph{$\mathcal{T}$-invariant} if $T\mu=\mu$ for each $T\in\mathcal{T}$.
Denote by $\mathcal{M}(X,\mathcal{T})$ the set of all $\mathcal{T}$-invariant elements in $\mathcal{M}(X)$. $\mu\in\mathcal{M}(X,\mathcal{T})$ is called \emph{ergodic} if $\mu(\bigcup_{T\in\mathcal{T}}TA)=0$ or $1$ for any $A\in \mathcal{B}_X$.
Denote by $\mathcal{M}^e(X,\mathcal{T})$ the set of all ergodic elements in $\mathcal{M}(X,\mathcal{T})$.
When the acting group $\mathcal{T}$ is amenable, $\mathcal{M}(X,\mathcal{T})\neq\emptyset$ and $\mathcal{M}(X)$, $\mathcal{M}(X,\mathcal{T})$ are convex compact metric spaces with weak$^*$-topology.

A partition of $X$ is a cover of $X$, whose elements are pairwise disjoint. Denote by $\mathcal{P}_X$ the set of all finite Borel partitions of $X$. Given $\alpha\in \mathcal{P}_X$, a $\mathcal{T}$-invariant sub-$\sigma$-algebra $\mathcal{A}\subseteq\mathcal{B}_X$ and $\mu\in\mathcal{M}(X)$, define
$$ H_{\mu}(\alpha|\mathcal{A})=\sum_{A\in \alpha}\int E(1_A|\mathcal{A})\log E(1_A|\mathcal{A})d\mu,$$
where $E(1_A|\mathcal{A})$ is the expectation of $1_A$ with respect to $\mathcal{A}$. Define
$$h_{\mu}(\mathcal{T},\alpha|\mathcal{A})=\lim_{n\to \infty}\frac{1}{|F_n|}H_{\mu}(\alpha_{F_n}|\mathcal{A}).$$
Once again one can deduce the existence of this limit and its independence of the sequence $\{ F_n\}_{n=1}^{+\infty}$
(see \cite{Kammeyer-Rudolph,Ornstein-Weiss}). When $\mathcal{A}=\{\emptyset,X\}$, we write $h_{\mu}(T,\alpha|\{\emptyset,X\})$ as $h_{\mu}(T,\alpha)$.
The \emph{measure-theoretic entropy of $(X,\mathcal{T},\mu)$} is defined by
$$h_{\mu}(X,\mathcal{T})=\sup_{\alpha\in\mathcal{P}_X}h_{\mu}(\mathcal{T},\alpha).$$

The variational principle between topological entropy and measure-theoretic entropy also holds for countable infinite discrete amenable group actions (see \cite{Ollagnier-Pinchon,Stepin-Tagi-Zade}):
$$h_{top}(X,\mathcal{T})=\sup_{\mu\in \mathcal{M}(X,\mathcal{T})}h_{\mu}(X,\mathcal{T})=\sup_{\mu\in \mathcal{M}^e(X,\mathcal{T})}h_{\mu}(X,\mathcal{T}).$$

For a $\mathcal{T}$-system, one has the following property (see \cite[Lemma 2.4]{H-Y-Z} or \cite[Theorem 6.1]{Lindenstrauss-Weiss}).

\begin{prop}\label{Entropy} Let $(X,\mathcal{T})$ be a $\mathcal{T}$-system, where $\mathcal{T}$ is a countable infinite discrete amenable group. Then for any $\alpha \in \mathcal{P}_X$ and $\mu\in \mathcal{M}(X,\mathcal{T})$, we have $$h_{\mu}(\mathcal{T},\alpha)=\inf_{\mathcal{F}\subset \mathcal{T},|\mathcal{F}|<\infty}\frac{H_{\mu}(\alpha_F)}{|F|}. $$
\end{prop}

\subsection{Extensions between measure preserving systems}
We say that a probability space $(X,\mathcal{B},\mu)$ is \emph{regular} if there exists a metric on $X$ suct that
$X$ is a compact metric space and $\mathcal{B}$ consists of all Borel subsets of $X$. In this paper, we always assume that probability spaces to be regular.
A \emph{measure preserving system} $(X,\mathcal{B},\mu,\mathcal{T})$ consists of a probability space $(X,\mathcal{B},\mu)$ and a group $\mathcal{T}$ acting on $X$ by transformations preserving measure $\mu$.
A measure preserving system $(X,\mathcal{B},\mu,\mathcal{T})$ is \emph{regular} if the underlying  probability space $(X,\mathcal{B},\mu)$ is regular.

A homomorphism of measure preserving systems $\pi:(X,\mathcal{B},\mu,\mathcal{T})\to (Y,\mathcal{D},\nu,\mathcal{T})$ is given by
a homomorphism $\pi\colon (X,\mathcal{B},\mu)\to (Y,\mathcal{D},\nu)$
satisfying
\begin{enumerate}
	\item $\pi^{-1}(A_1\cup A_2)=\pi^{-1}(A_1)\cup\pi^{-1}(A_2)$,
	$A_1,A_2\in\widetilde{\mathcal{D}}$,
	\item $\pi^{-1}(\widetilde Y\setminus A)=X\setminus \pi^{-1}(A)$,
	$A\in\widetilde{\mathcal{D}}$,
	\item $\mu(\pi^{-1}(A))=\nu(A)$, $A\in\widetilde{\mathcal{D}}$,
	\item $\pi^{-1}(T^{-1}A)=T^{-1}(\pi^{-1}(A))$,
	 $A\in\widetilde{\mathcal{D}}$, $T\in\mathcal{T}$,
\end{enumerate}
where $\widetilde{\mathcal{D}}$ is the abstract $\sigma$-algebra
consisting of equivalence classes of sets in $\mathcal{D}$ (mod null sets).
In this case we say that $(X,\mathcal{B},\mu,\mathcal{T})$  is \emph{an extension of $(Y,\mathcal{D},\nu,\mathcal{T})$} or that $(Y,\mathcal{D},\nu,\mathcal{T})$  is  a \emph{factor  of  $(X,\mathcal{B},\mu,\mathcal{T})$}, and $\pi$ is a \emph{factor map}.
The following result is well known (see e.g. \cite[Theorem 5.8]{F})
\begin{thm}\label{F5.8}
Let $(X,\mathcal{B}_X,\mu,\mathcal{T})$ be a regular measure preserving system  and $\pi:(X,\mathcal{B}_X,\mu,\mathcal{T})\to(Y,\mathcal{B}_Y,\nu,\mathcal{T})$ be a factor map.
Then there exists a measurable map from $Y$ to $\mathcal{M}(X)$ which we shall denote $y\to \mu_y$ which satisfies:
\begin{enumerate}
  \item for every $f\in L^1(X,\mathcal{B},\mu)$, $f\in L^1 (X,\mathcal{B},\mu_y)$ for $\nu$-a.e. $y\in Y$, and
  $$E(f |Y)(y)= \int fd\mu_y$$
  for $\nu$-a.e. $y\in Y$;
  \item$\int \{\int fd\mu_y\}d\nu(y)=\int fd\mu $ for every $f\in L^1(X,\mathcal{B},\mu)$.
\end{enumerate}
\end{thm}
We shall write $\nu=\int\mu_yd\nu$ and refer to this as \emph{the disintegration  of $\mu $   with  respect to  the  factor   $(Y,\mathcal{D},\nu)$} (or \emph{the disintegration of $\mu$ over  $\nu $}). For each $S\in\mathcal{T}$ and for almost every $y\in Y$, $\mu_{Sy}=S\mu_{y}$.

Let $(X_1,\mathcal{B}_1,\mu_{1})$ and $(X_2,\mathcal{B}_2,\mu_2)$ be two regular measure spaces and
$$\pi_1:(X_1,\mathcal{B}_1,\mu_1)\to (Y,\mathcal{D},\nu), \ \ \ \pi_2:(X_2,\mathcal{B}_2,\mu_2)\rightarrow (Y,\mathcal{D},\nu) $$
are two extensions of the same space $(Y,\mathcal{D},\nu)$.
The measure space $(X_1\times X_2,\mathcal{B}_1\times\mathcal{B}_2,\mu_1\times_Y\mu_2)$ is called \emph{the  product  of  $(X_1,\mathcal{B}_1,\mu_1)$  and  $(X_2,\mathcal{B}_2,\mu_2)$  relative  to $(Y,\mathcal{D},\nu)$ } and is denoted by $X_1\times_Y X_2$, where $\mu_1\times_Y\mu_2$ is a measure on $(X_1\times X_2,\mathcal{B}_1\times \mathcal{B}_2)$ defined by:
$$(\mu_1\times_Y\mu_2)(A)=\int\mu_{1,y}\times\mu_{2,y}(A)d\nu(y)$$
for $A\in \mathcal{B}_1\times \mathcal{B}_2$, and $\mu_i=\int \mu_{i,y}d\nu(y)$ are the disintegrations of $\mu_i$ over $\nu$, for $i=1,2$.

Let $(X,\mathcal{B},\mu,\mathcal{T})$
be a regular measure preserving system and $\pi:(X,\mathcal{B},\mu,\mathcal{T})\to (Y,\mathcal{D},\nu,\mathcal{T})$ be a factor map.
We say that $\pi$ is an \emph{ergodic  extension   of  $(Y,\mathcal{D},\nu,\mathcal{T})$  relative  to $T\in\mathcal{T}$}
if the only $T$-invariant sets of $\mathcal{B}$ are images (modulo null sets) under $\pi^{-1}$ of $T$-invariant sets of $\mathcal{D}$,
and a \emph{weakly mixing extension  of $(Y,\mathcal{D},\nu,\mathcal{T})$ relative to $T\in \mathcal{T}$} if
$(X\times X,\mathcal{B}\times\mathcal{B},\mu\times_Y\mu,\mathcal{T})$ is an ergodic extension of $(Y,\mathcal{D},\nu,\mathcal{T})$ relative to $T$.
We say that $\pi$ is an \emph{ ergodic extension  (or a  weakly  mixing extension ) of $(Y,\mathcal{D},\nu,\mathcal{T})$}
if $\pi$ is an ergodic extension (or a weakly mixing extension) of $(Y,\mathcal{D},\nu,\mathcal{T})$ relative to every $T\in\mathcal{T}\setminus \{\theta_{\mathcal{T}}\}$.

Furstenberg proved the following proposition (see \cite[Proposition 6.4]{F}).
\begin{prop}\label{TWM}
Let $(X,\mathcal{B},\mu,\mathcal{T})$ be a regular measure preserving system, and let $\pi:(X,\mathcal{B},\mu,\mathcal{T})\to (Y,\mathcal{D},\nu,\mathcal{T})$ be a factor map.
If $\pi$ is a weakly mixing extension of $(Y,\mathcal{D},\nu,\mathcal{T})$,
then $\widetilde{\pi}=\pi\circ \pi_1:(X\times X,\mathcal{B}\times\mathcal{B},\mu\times_Y\mu,\mathcal{T})\to (Y,\mathcal{D},\nu,\mathcal{T})$ is also a weakly mixing extension of $(Y,\mathcal{D},\nu,\mathcal{T})$, where $\pi_1:X\times X\to X$ is the projection to the first coordinate.
\end{prop}

Let $(X,\mathcal{T})$ be a $\mathcal{T}$-system, where $\mathcal{T}$ is a countable infinite discrete amenable group. For $\mu\in\mathcal{M}(X,\mathcal{T})$ and $\mathcal{T}$-invariant sub-$\sigma$-algebra $\mathcal{A}$ of $\mathcal{B}_X$, denote
$$P_X^{\mu}(\mathcal{T}|\mathcal{A})=\{A\in \mathcal{B}_X: h_{\mu}(\mathcal{T},\{A,X\setminus A\}|\mathcal{A})=0\}.$$
It follows from \cite[Lemma 1.1]{Glasner-Thouvenot-Weiss} or \cite[Theorem 3.1]{Huang-Xu-Yi} that $P_X^{\mu}(\mathcal{T}|\mathcal{A})$ must be a $\mathcal{T}$-invariant sub-$\sigma$-algebra of $\mathcal{B}_X$ containing $\mathcal{A}$. We call this $\sigma$-algebra the \emph{Pinsker $\sigma$-algebra of $(X,\mathcal{B}_X,\mathcal{T},\mu)$ relative to $\mathcal{A}$}, and the corresponding factor the \emph{relative Pinsker factor}.
When $\mathcal{A}$ is the trivial $\sigma$-algebra we get the \emph{Pinsker algebra} and \emph{Pinsker factor of $X$}
and we denote this $\sigma$-algebra by $P_X^{\mu}(\mathcal{T})$.
The following theorem is a classic result (see for example \cite[Theorem 0.4]{Glasner-Thouvenot-Weiss} or \cite[Lemma 4.2]{Huang-Xu-Yi}).

\begin{thm}\label{WM}
Let $(X,\mathcal{T})$ be a $\mathcal{T}$-system, where $\mathcal{T}$ is a countable infinite discrete amenable group, $\mu\in\mathcal{M}^e(X,T)$,
and $\pi:(X,\mathcal{B}_X,\mu,\mathcal{T})\to (Z,\mathcal{B}_Z,\nu,\mathcal{T})$ be the factor map to the Pinsker factor of $(X,\mathcal{B}_X,\mu,\mathcal{T})$.
Assume that $\pi_1:X\times X\to X$ is the  projection to the first coordinate and $\widetilde{\pi}=\pi\circ\pi_1$.
Then  $P_{X\times X}^{\mu\times_Z\mu}(\mathcal{T}|\widetilde{\pi}^{-1}(\mathcal{B}_Z))=\widetilde{\pi}^{-1}(\mathcal{B}_Z)$ $(\bmod \mu\times_Z\mu)$.
\end{thm}

\begin{prop}\label{Pinsker}
Let $(X,\mathcal{T})$ be a $\mathcal{T}$-system, where $\mathcal{T}$ is a countable infinite discrete amenable group, $\mu\in\mathcal{M}^e(X,T)$,
and $\pi:(X,\mathcal{B}_X,\mu,\mathcal{T})\to (Z,\mathcal{B}_Z,\nu,\mathcal{T})$ be the factor map to the Pinsker factor of $(X,\mathcal{B}_X,\mu,\mathcal{T})$.
Then $\pi$ is a weakly mixing extension.
\end{prop}
\begin{proof}
Let  $\pi_1:X\times X$ be the  projection to the first coordinate and $\widetilde{\pi}=\pi\circ \pi_1:(X\times X, \mathcal{B}_X\times\mathcal{B}_X,\mu\times_Z\mu,\mathcal{T})\to (Z,\mathcal{B}_Z,\nu,\mathcal{T})$.
	
For any $T\in\mathcal{T}\setminus{\theta_{\mathcal{T}}}$, we shall show that $\widetilde{\pi}$ is an ergodic extension of $(Z,\mathcal{B}_Z,\nu,\mathcal{T})$ relative to $T$. Suppose $E \in \mathcal{B}_X\times \mathcal{B}_X$ such that $TE=E$. Let $\alpha=\{E,X\times X\setminus E\}$, and $F_n=\{T,T^2,\dotsc ,T^n\}\subset\mathcal{T}$, $n\in\N$. Then for any $n\in\N$, $\alpha_{F_n}=\alpha$. By Proposition \ref{Entropy} we have
	$$h_{\mu\times_Z\mu}(\mathcal{T},\alpha|\widetilde{\pi}^{-1}(\mathcal{B}_Z))\leq h_{\mu\times_Z\mu}(\mathcal{T},\alpha)=\inf_{F\subset\mathcal{T},|\mathcal{F}|<\infty}\frac{H_{\mu\times_Z\mu}(\alpha_{F})}{|F|}\leq \frac{H_{\mu\times_Z\mu}(\alpha_{F_n})}{|F_n|}, $$
	for any $n\in N$. This implies $h_{\mu\times_Z\mu}(\mathcal{T},\alpha\ |\ \widetilde{\pi}^{-1}(\mathcal{B}_Z))=0$, thus
	$$E\in P_{X\times X}^{\mu\times_Z\mu}(\mathcal{T}\ |\ \widetilde{\pi}^{-1}(\mathcal{B_Z}))=\widetilde{\pi}^{-1}(\mathcal{B_Z})$$
	by Theorem \ref{WM}. This finishes the proof.
\end{proof}

\section{$\Delta$-weakly mixing set }
In this section, we assume that $\mathcal{T}$ is a countable   torsion-free discrete group. We will present some properties of $\Delta$-transitive subsets and $\Delta$-weakly mixing subsets of a $\mathcal{T}$-system, by partially following the arguments in \cite[section 3]{H-L-Y-Z}.
\begin{prop}\label{weakly mixng} Let $(X,\mathcal{T})$ be a $\Delta$-weakly mixing $\mathcal{T}$-system, then $X$ is perfect and $(X,\mathcal{T})$ is weakly mixing.
\end{prop}
\begin{proof} Recall that we require $|E|\geq 2$ in the definition of $\Delta$-weakly mixing subset $E$ of $(X,\mathcal{T})$. Thus $|X|\geq2$ if $(X,\mathcal{T})$ is $\Delta$-weakly mixing. Suppose $X$ is not perfect, there exists an non-empty open set $U$ of $X$ such that $U=\{u\}$ for some $u\in X$. Now we pick $T\in\mathcal{T}\setminus\{\theta_{\mathcal{T}}\}$, non-empty open subsets $V_1,V_2$ of $X$, such that $V_1\cap V_2=\emptyset$. Since $X^2$ is a $\Delta$-transitive subset of $(X^2,\mathcal{T})$ and $\{(u,u)\}=U\times U$ is an open subset of $X^2$, one has $\{T^n(u,u):n\in\N\}$ is dense in $X^2$. Hence there exists $n_0\in\N$, such that $T^{n_0}u\in V_1$ and $T^{n_0}u\in V_2$, which contradicts with $V_1\cap V_2=\emptyset$. Thus $X$ is perfect.

For any non-empty open subsets $U_1,U_2$ and $V_1,V_2$ of $X$, we pick distinct $T_1,T_2\in\mathcal{T}\setminus\{\theta_{\mathcal{T}}\}$. Since $(X,\mathcal{T})$ is $\Delta$-weakly mixing, there exists $x=(x_1,x_2)\in X^2$ such that
$$\overline {orb_{+}((x,x), T_1\times T_2)}=X^2\times X^2.$$
So we can find $n\in\N$ such that $T_1^nx_i\in U_i$ and $T_2^nx_i\in V_i$ for $i=1,2$, this implies
$$T_{1}^nx_1\in T_1^nT_2^{-n}V_1\cap U_1\neq\emptyset, \ \ T_1^nx_2\in T_1^nT_2^{-n}V_2\cap U_2\neq\emptyset.$$
Thus $(X,\mathcal{T})$ is weakly mixing. This finishes our proof.
\end{proof}

\begin{rem}\label{remark 1} Similarly, we can obtain that $E$ is perfect if it is a $\Delta$-weakly mixing subset of a $\mathcal{T}$-system $(X,\mathcal{T})$.
\end{rem}

Let $(X,\mathcal{T})$ be a $\mathcal{T}$-system.
For any $d\in\N$, $T_1,T_2,\dotsc ,T_d\in\mathcal{T}$, non-empty subsets $V$ and $U_1,U_2,\dotsc ,U_d$ of $X$, we define
$$ N(V;U_1,\dotsc ,U_d \mid T_1,T_2,\dotsc ,T_d):=\{n \in \mathbb{N} \colon V \cap T_1^{-n}U_1\cap\dotsc \cap T_d^{-n}U_d \neq \emptyset \}.$$
\begin{lem}\label{TR}
Let $(X,\mathcal{T})$ be a $\mathcal{T}$-system and $E$ be a closed subset of $X$ with $|E|\geq 2$.
Then $E$ is $\Delta$-transitive if and only for any integer $d\geq 1$, pairwise distinct
$T_1,T_2,\dotsc ,T_d\in \mathcal{T}\setminus\{\theta_{\mathcal{T}}\}$, and non-empty open subsets $V,U_1,U_2,\dotsc ,U_d$ of $X$ intersecting $E$, one has
$$N(V\cap E;U_1,\dotsc ,U_d \mid T_1,T_2,\dotsc ,T_d)\neq \emptyset. $$
\end{lem}
\begin{proof}
Necessity. If $E$ is a $\Delta$-transitive subset of $(X,\mathcal{T})$, then for any $d\geq 1$, pairwise distinct $T_1,T_2,\dotsc ,T_d\in\mathcal{T}\setminus\{\theta_{\mathcal{T}}\}$, and non-empty open subsets $V,U_1,U_2,\dotsc ,U_d$ of $X$ intersecting $E$, there exists $x\in E\cap V$ such that for the diagonal $d$-tuple $(x,x,\dotsc ,x)$ one has
$$\overline{orb_{+}((x,x,\dotsc ,x),T_1\times T_2\times\dotsb \times T_d)}\supseteq E^d.$$
This implies there exists $n_0\in \N$ such that $T_i^{n_0}x\in U_i$, $i=1,2,\dotsc ,d$. Thus $n_0\in N(V\cap E;U_1,\dotsc ,U_d \mid T_1,T_2,\dotsc ,T_d)\neq \emptyset.$

Sufficiency. Let $\mathcal{W}$ be a countable topological base of $X$ and
$$\mathcal{U}=\{U\in\mathcal{W}:U\cap{E}\neq\emptyset\}.$$
Then $\mathcal{U}$ is also a countable set. Since $\mathcal{T}$ is a countable group, we can enumerate it as
$\{\theta_{\mathcal{T}}, T_1,T_2,\dotsc \}$. Let
$$A:=\biggl(\bigcap_{d=1}^{\infty}\bigcap_{\{U_1,U_2,\dotsc ,U_d\}\subset \mathcal{U}}\bigcup_{n=1}^{\infty}(T_1^{-n}U_1\cap T_2^{-n}U_2\cap\dotsc \cap T_d^{-n}U_d)\biggr)\cap E.$$
Then $A$ is a residue subset of $E$. For any $d\geq1$,  pairwise distinct elements $T_1',T_2',\dotsc ,T_d'\in \mathcal{T}\setminus\{\theta_{\mathcal{T}}\}$, and non-empty open subsets $U_1',U_2',\dotsc ,U_d'$ of $X$ intersecting $E$. Choose any $v\in A$, $U_{h_i}\in \mathcal{U}$ such that $U_{h_i} \subset U_i'$ for $i=1,2,\dotsc ,d$, and an integer $L$ large enough such that $\{T_1',T_2',\dotsc ,T_d'\} \subset \{ T_1,T_2,\dotsc ,T_L\} $. Without loss of generality, we can assume $T_i=T_i'$ for $i=1,2,\dotsc ,d$. Since
$$v\in \bigcup_{n=1}^{\infty}(T_1^{-n}U_{h_1}\cap T_2^{-n}U_{h_2}\cap\dotsc \cap T_d^{-n}U_{h_d}\cap\dotsc \cap T_L^{-n}U_{h_L})$$
where $U_{h_{d+1}},U_{h_{d+2}},\dotsc ,U_{h_L}$ are any $L-d$ non-empty open subsets in $\mathcal{U}$, there exists $k\in\mathbb{N}$ such that $T_i^kv\in U_{h_i}\subset U_i'$ for $i=1,2,\dotsc ,d$. This implies
$$\overline {orb_{+}\{(v,v,..,v); T_1'\times T_2'\times\dotsc \times T_d' \}}\supseteq E^d.$$
Thus $E$ is a $\Delta$-transitive subset of $(X,\mathcal{T})$.
\end{proof}

Using Lemma \ref{TR} we have the following result.
\begin{prop}\label{d-WM}
Let $(X,\mathcal{T})$ be a $\mathcal{T}$-system and $E$ be a closed subset of $X$ with $|E|\geq2$.  Then $E$ is a $\Delta$-weakly mixing subset of $X$ if and only if for any $d\geq1$, pairwise distinct $T_1,T_2,\dotsc ,T_d\in \mathcal{T}\setminus \{\theta_{\mathcal{T}} \}$, and non-empty open subsets
$ U_1,U_2,\dotsc ,U_d$ and $V_1,V_2,\dotsc ,V_d$ of $X$ intersecting $E$, one has
$$ \bigcap_{s\in \{1,2,\dotsc ,d \}^{d+1}}N(V_{s(1)}\cap E;U_{s(2)},\dotsc ,U_{s(d+1)} \mid T_1,T_2,\dotsc ,T_d)\neq \emptyset.$$
\end{prop}
\begin{proof}To prove the sufficiency, we shall to show $E^n$ is a $\Delta$-transitive subset of $(X^n,\mathcal{T})$ for any fixed $n\in\mathbb{N}$. For any $d\geq 1$, distinct $T_1,T_2,\dotsc ,T_d\in \mathcal{T}\setminus\{\theta_{\mathcal{T}}\} $, and non-empty open subsets $V_i,U_{ij}$ of $X$ intersecting $E$ for $i=1,2,\dotsc ,n$ and $j=1,2,\dotsc ,d$. Let $N=nd$ and choose pairwise distinct $T_1',T_2',\dotsc ,T_N'\in \mathcal{T}\setminus \{ \theta_{\mathcal{T}} \}$ such that $T_i=T_i'$ for $i=1,2,\dotsc ,n$, since $|\mathcal{T}|=\infty$. We rewrite $\{U_{ij}; \ i=1,2,\dotsc ,n; \ j=1,2,\dotsc ,d \}$ as
$\{U_1,U_2,\dotsc ,U_N \}$, and let $V_{kd+j}'=V_j$ for $k=0,1,\dotsc ,n-1$ and $j=1,2,\dotsc ,d$. Then one has
$$ \bigcap_{s\in \{1,2,\dotsc ,N \}^{N+1}}N(V_{s(1)}'\cap E;U_{s(2)},\dotsc ,U_{s(N+1)}| \ T_1',T_2',\dotsc ,T_N')\neq \emptyset. $$
In particular, there exists $L\in \mathbb{N}$ such that $(V_i\cap E)\cap(\bigcap_{j=1}^{d}T_j^{-L}U_{ij})\neq \emptyset$ for any $i=1,2,\dotsc ,n$. Then $E^n$ is a $\Delta$-transitive subset of $(X^n,\mathcal{T})$ by Lemma \ref{TR}.

Necessity. Suppose $E$ is a $\Delta$-weakly mixing subset of $(X,\mathcal{T})$. Fix $d\geq 1$ and let $L=|\{1,2,\dotsc d\}^{d+1}|$. Then $E^L$ is a $\Delta$-transitive subset of $(X^L,\mathcal{T})$. We can rewrite $\{1,2,\dotsc ,d\}^{d+1}=\{s^1,s^2,..,s^L\}$. For pairwise distinct $T_1,T_2,..,T_d\in\mathcal{T}\setminus\{\theta_{\mathcal{T}}\}$, and non-empty open subsets $U_1,U_2,\dotsc ,U_d$ and $V_1,V_2,..,V_d$ of $X$ intersecting $E$, there exists $x_k\in V_{s^k(1)}\cap E$ for $k=1,2,\dotsc L$ such that $L$-tuple $x=(x_1,x_2,\dotsc ,x_L)\in E^L$ satisfies
$$\overline{orb_{+}((x,x,\dotsc ,x),T_1\times T_2\times\dotsc T_d)}\supseteq E^L\times E^L\times\dotsc \times E^L\ (d\text{-times}).$$
Thus there exists $n_0\in\N$  such that $T_i^{n_0}x_k\in U_{s^k(i+1)}$ for $k=1,2,\dotsc ,L$ and $i=1,2,\dotsc ,d$, which implies that
$$n_0\in \bigcap_{s\in \{1,2,\dotsc ,d \}^{d+1}}N(V_{s(1)}\cap E;U_{s(2)},\dotsc ,U_{s(d+1)}| \ T_1,T_2,\dotsc ,T_d)\neq \emptyset.$$
This ends our proof.
\end{proof}

\section{Entropy and $\Delta$-weakly mixing subsets of nilpotent group actions }
In this section we will introduce the concept and some properties of nilpotent group. Finally, we will prove Theorem \ref{thm-A} by partially following from the argument in the proof of Theorem C in \cite{H-L-Y-Z} .
\subsection{Nilpotent group-polynomial}
A group $\mathcal{T}$ with the unit $\theta_{\mathcal{T}}$ is called \emph{nilpotent} if it has a finite sequence of normal subgroups
(\emph{a finite  central series}):
$\{\theta_{\mathcal{T}}\}\ =\mathcal{T}_0 \subset \mathcal{T}_1\subset\dotsc \subset \mathcal{T}_t=\mathcal{T} $, such that $[\mathcal{T}_i, \mathcal{T}]\subset \mathcal{T}_{i-1}$ for $i=1,2,\dotsc t$, where $[\mathcal{T}_i,\mathcal{T}]$ denotes the subgroup generated by $\{[T,S]=T^{-1}S^{-1}TS: T\in \mathcal{T}_i, \ S\in \mathcal{T}\} $. Any finitely generated nilpotent group is a factor of finitely generated torsion-free nilpotent group, thus every representation of a finitely generated nilpotent group can be lifted to a representation of a finitely generated torsion-free nilpotent group.

If $\mathcal{T}$ is a finitely generated nilpotent torsion-free nilpotent group, then there exists a subset $\{S_1, S_2,\dotsc ,S_s\}$ of $\mathcal{T}$ (\emph{Maccev basis} of $\mathcal{T}$) such that every element $T \in \mathcal{T}$ can be uniquely represented in the form
$$ T=S_1^{r_1(T)}S_2^{r_2(T)}\dotsc S_s^{r_s(T)},$$
where the mapping $r:\mathcal{T}\rightarrow Z^s$:
$$r(T)=(r_1(T),r_2(T),\dotsc ,r_s(T)),$$
such that there exist polynomial mappings $\phi : \mathbb{Z}^{s+1}\rightarrow \mathbb{Z}^s$, for any $T\in \mathcal{T}$
$$r(T^n)=\phi(r(T),n).$$

The group of $\mathcal{T}$-polynomials $P\mathcal{T}$, is the minimal subgroup of the group $\mathcal{T}^{\mathbb{Z}}$ of the mappings $\mathbb{Z}\rightarrow \mathcal{T}$ which contains constant sequences and is closed with respect to raising to integral polynomial powers: if $g,h \in P\mathcal{T}$, and $p$ is an integral polynomial (taking integer values at the integers), then $gh,\ g^p \in P\mathcal{T}$, where $gh(n)=g(n)h(n),\ g^p(n)=g(n)^{p(n)}$, $n\in \mathbb{Z}$. Then
$$\Phi_T: n\rightarrow T^n=S_1^{\phi(r(T),n)_1}S_2^{\phi(r(T),n)_2}\dotsc S_s^{\phi(r(T),n)_s}$$
is $\mathcal{T}$-polynomial, for any $T\in \mathcal{T}$.

\subsection{PET induction}
In \cite{Leibman}, the author introduced the weight $\omega(g)$ of a $\mathcal{T}$-polynomial, then for any \emph{system} (finite subset of $P\mathcal{T}$) $\mathcal{A}$ the weight vector $\omega(A)$ is defined. The set of weight vectors is well ordered, we say the system $\mathcal{A}'$ \emph{percedes} $\mathcal{A}$ if $\omega(\mathcal{A})$ grater than $\omega(\mathcal{A}')$. The \emph{PET-induction} is an induction on the well ordered set of systems, that is, if a statement is true for the system $\{\theta_{\mathcal{T}}\}$ and one can show that it holds for a system $\mathcal{A}$ from the assumption that it is true for all systems preceding $\mathcal{A}$, then we can assert this statement holds for all systems.

Using PET-induction, Leibman proved the following proposition (see \cite[Corollary 11.7]{Leibman}).
\begin{prop}\label{Leibman} Let $(X,\mathcal{T})$ be a $\mathcal{T}$-system, where $\mathcal{T}$ is a finitely generated torsion-free discrete nilpotent group and $\mu\in \mathcal{M}(X,\mathcal{T})$. If $\pi: (X,\mathcal{B}_X, \mu,\mathcal{T})\to (Y,\mathcal{B}_Y,\nu,\mathcal{T})$ is a weakly mixing extension and $\mu=\int \mu_yd\nu$ is the decomposition of the measure $\mu$ over $\nu$, then
\begin{align}\label{D-lim-e}
D-\lim_{n \to \infty} \biggl\{\int\prod_{i=1}^df_i\circ T_i^nd\mu_y-\prod_{i=1}^{d}\int f_i\circ T_i^nd\mu_y \biggr\}=0\ in\ L^1(Y),
\end{align}
for any $d\geqslant 1$, pairwise distinct $T_1,T_2,\dotsc ,T_d \in \mathcal{T}$ and $f_1,f_2,\dotsc ,f_d\in L^{\infty}(\mu)$.
\end{prop}

Here, for a sequence of points $\{z_n\}_{n=1}^{\infty}$ in a topological space $Z$ and $z\in Z$,
$D$-$\lim_{n\to+\infty}z_n=z$ means that $\{z_n\}_{n=1}^{\infty}$ converges to $z$ in density, that is, for every neighborhood $V$ of $z$ in $Z$, $z_n\in V$ for all $n$ except a set of zero density.
It is clear that if (\ref{D-lim-e}) holds then for any $\varepsilon>0$ and $0<\delta <1$, the collection of $n$ satisfying
$$\nu\biggl(\biggl\{y\in Y: \biggl|\int \prod_{i=1}^df_i\circ T_i^nd\mu_y-\prod_{i=1}^{d}\int f_i\circ T_i^nd\mu_y\biggr| <\varepsilon \biggr\}\biggr) >1-\delta. $$
  has density one.
\subsection{Proof of Theorem {\ref{thm-A}}}
The following lemma (see \cite[Theorem NM']{Leibman}) and propositions will be used in the proof of Theorem \ref{thm-A}.

\begin{lem}\label{Leibman2}
Let $(Y,\mathcal{D},\nu,\mathcal{T})$ be a measure preserving system, where $\mathcal{T}$ is a countable discrete nilpotent group.
Then for any $d\in\N$, $A\in\mathcal{D}$ with $\nu(A)>0$ and $T_1,T_2,\dotsc ,T_d\in\mathcal{T}$, one has
$$\liminf_{N \to \infty}\frac{1}{N}\sum_{n=0}^{N-1}\nu\biggl(\bigcap_{i=1}^{d}T_i^{-n}A\biggr)>0.$$
\end{lem}

\begin{prop}\label{PM}
Let $(X,\mathcal{T})$ be a $\mathcal{T}$-system, where $\mathcal{T}$ is a finitely generated torsion-free discrete nilpotent group and $\mu\in\mathcal{M}(X,\mathcal{T})$. Let $\pi:(X,\mathcal{B}_X, \mu,\mathcal{T})\to (Y,\mathcal{B}_Y,\nu,\mathcal{T})$ be a weakly mixing extension and $\mu=\int \mu_yd\nu$ be the decomposition of the measure $\mu$ over $\nu$.
For any positive integers $k$ and $M$, pairwise distinct $T_1,T_2,\dotsc ,T_k\in \mathcal{T}\setminus\{\theta_{\mathcal{T}}\}$,
if $A_1,A_2,\dotsc ,A_M\in \mathcal{B}_X$  satisfies that
$$\Omega:=\{ y\in Y: \mu_y(A_i)>0,\ for\ all\ i=1,2,\dotsc ,M \}$$
has positive $\nu$-measure, then
we can find $L\in\N$ and $c>0$ such that
$$\Omega':=
\biggl\{ y\in Y: \mu_y\biggl(A_{s(1)}\cap \bigcap_{i=1}^{k}T_i^{-L}A_{s(i+1)}\biggr)>c, \text{ for all } s\in \{1,2,\dotsc ,M\}^{k+1}\biggr \} $$
has positive $\nu$-measure.
\end{prop}
\begin{proof}
For every $p\in\N$, let
\[\Omega_p=\Bigl\{y\in Y: \mu_y(A_i)>\tfrac{1}{p}, \text{ for all }i=1,2,\dotsc ,M \Bigr \}.\]
It is clear that $\Omega=\bigcup_{p=1}^{\infty}\Omega_p $.
As $\nu(\Omega)>0$, there exists some $p\in \N$ such that $\nu(\Omega_p)>0$.
By Lemma \ref{Leibman2} there exists $\lambda>0$ such that
$$ \liminf_{N \rightarrow \infty}\frac{1}{N}\sum_{n=0}^{N-1}
\nu\biggr(\Omega_p\cap\bigcap_{i=1}^{k}T_i^{-n}\Omega_p\biggr)> \lambda,$$
then $E:=\{ n\in\N: \nu(\Omega_p\cap\bigcap_{i=1}^{k}T_i^{-n}\Omega_p)> \lambda \}$ has positive lower density.

Fix $0<\varepsilon<\frac{1}{p^{k+1}}$ and $0<\delta<\frac{\lambda}{M^{k+1}}$.
For any $s\in\{1,2,\dotsc ,M\}^{k+1}$, let
\[A_s^n=A_{s(1)}\cap \bigcap_{i=1}^{k}T_i^{-n}A_{s(i+1)}\]
and $F_s$ be the collection of $n$ such that
\begin{align}\label{e-2}
\nu\biggl(\biggl\{y\in Y\colon \biggl| \mu_y(A_s^n)-\mu_y(A_{s(1)})
\prod_{i=1}^k\mu_{(T_i^Ly)}(A_{s(i+1)})\biggr|<\varepsilon \biggr\}\biggr)>1-\delta.
\end{align}
Then by Proposition \ref{Leibman} $D(F_s)=1$ for any $s\in\{1,2,\dotsc ,M\}^{k+1}$. Thus $F:=\bigcap_{s\in\{1,2,\dotsc ,M\}^{k+1}}F_s$ has density $1$ and $E\cap F\neq\emptyset$. We can pick $L\in E\cap F$ and let
\begin{align*}
V=\biggl\{y\in Y\colon :& \biggl|\mu_y(A_s^L)-\mu_y(A_{s(1)})\prod_{i=1}^k\mu_{(T_i^Ly)}(A_{s(i+1)})\biggr| <\varepsilon,\\
&\qquad\qquad \qquad\quad\text{ for all }s\in \{1,2,\dotsc ,M \} ^{k+1}\biggr\}
\end{align*}
and
\[H=V\cap \Omega_p\cap T_1^{-L}\Omega_p\cap \dotsb \cap T_k^{-L}\Omega_p.\]
Then by (\ref{e-2}) and $L\in E$ one has
$$\mu(V)>1-M^{k+1}\delta \text{ and } \nu(H)>\lambda-M^{k+1}\delta>0.$$
Now we pick $0<c<\frac{1}{p^{k+1}}-\varepsilon$ such that for any
$y\in H$, we have
$$\mu_y(A_s^L)>\mu_y(A_{s(1)})\prod_{i=1}^{k}\mu_{(T_i^Ly)}(A_{s(i+1)})-\varepsilon>\frac{1}{p^{k+1}}-\varepsilon>c,$$
for any $s\in\{1,2,\dotsc ,M\}^{k+1}$. This finishes our proof.
\end{proof}

Now let's prove Theorem \ref{thm-A}, by partially following from the arguments in the proof of Theorem $C$ in \cite{H-L-Y-Z}.
\begin{proof}[Proof of the Theorem \ref{thm-A}]
Since $h_{top}(X,\mathcal{T})>0$, there exists $\mu\in\mathcal{M}^e(X,\mathcal{T})$ such that $h_{\mu}(X,\mathcal{T})>0$.
Let $\pi:(X,\mathcal{B}_X,\mu,\mathcal{T})\to(Z,\mathcal{B}_Z,\nu,\mathcal{T})$ be the factor map to the Pinsker factor of $(X,\mathcal{B}_X,\mu,\mathcal{T})$.
By Proposition \ref{Pinsker}, we know $\pi$ is a weakly mixing extension.
Let $\mu=\int\mu_zd\nu$ be the decomposition of the measure $\mu$ over $\nu$.
Since $\mathcal{T}$ is countable,
we can enumerate $\mathcal{T}$ as $\{\theta_{\mathcal{T}},T_1,T_2,\dotsc \}$.

Let $\lambda=\mu\times_Z\mu$. Then by Proposition \ref{TWM}, $\widetilde{\pi}:=\pi_1\circ\pi:(X\times X,\mathcal{B}_X\times \mathcal{B}_X,\lambda,\mathcal{T})\to (Z,\mathcal{B}_Z,\nu,\mathcal{T})$ is also a weakly mixing extension, where $\pi_1:X\times X\to X$ is the projection to the first coordinate. Moreover, $\lambda(\Delta_X)=0$, where $\Delta_X=\{(x,x): x\in X\}$ (see e.g. \cite[Lemma 4.3]{Huang-Xu-Yi}). Then we can pick $(x_1,x_2)\in \supp(\lambda)\setminus\triangle_X$, and choose disjoint closed neighborhood $W_i$ of $x_i$ such that diam($W_i$)$<\frac{1}{2}$ for $i=1,2$ and
$$\lambda(W_1\times W_2)=\int{\mu_z}(W_1)\mu_z(W_2)d\nu(z)>0.$$

Let $\Omega=\{z\in Z: \mu_z(W_i)>0 \text{ for }i=1,2 \}$.
Then $\nu(\Omega)>0$.
We can find $c_1>0$  such that
$\Omega_1:=\{z\in Z: \mu_z(W_i)>c_1\text{ for }i=1,2\}$
has positive $\nu$-measure. Now we denote
$\mathscr{E}_1=\{1,2\}$,
$\mathscr{E}_2=\mathscr{E}_1\times \mathscr{E}_1$, $\dotsc$, $\mathscr{E}_k=\mathscr{E}_{k-1}\times\mathscr{E}_{k-1}\times\dotsc \times\mathscr{E}_{k-1}$ ($k$-times) for any $k\geq 3$.
Let $A_i=W_i$ for $i\in\mathscr{E}_1$. Then by induction and Proposition \ref{PM} we can construct a non-empty closed subset $A_\sigma$ of $X$ for each $\sigma \in \mathscr{E}_k$, $k\in\N$ with the following properties:
\begin{enumerate}
  \item for any $k>1$, there exists $n_k\in\N$ , and a non-empty closed subset $A_{\sigma}$ of $X$ for any $\sigma=(\sigma(1),\sigma(2),\dotsc ,\sigma(k))\in \mathscr{E}_{k}$, where $\sigma(i)\in \mathscr{E}_{k-1}, i=1,2,\dotsc ,k$, such that
$$A_\sigma\subset A_{\sigma(1)}\cap T_1^{-n_k}A_{\sigma(2)}\cap \dotsc  \cap T_{k-1}^{-n_k}A_{\sigma(k)};$$
  \item diam$( A_{\sigma}) < 2^{-k}$, for all $\sigma\in \mathscr{E}_{k}$, $k\in\N$;
  \item for any $k\in\N$, there exists $c_k>0$ such that
$$\{z\in Z: \mu_z(A_{\sigma})>c_k, \text{ for all }\sigma\in \mathscr{E}_{k}\}$$
has positive $\nu$-measure.
\end{enumerate}
Let $A=\bigcap_{k=1}^{\infty}\bigcup_{\sigma\in \mathscr{E}_{k}}A_{\sigma}$.
Now we shall show that $A$ is a $\Delta$-weakly mixing subset of $(X,\mathcal{T})$. Note that for any given $k\in\N$,
$\{A_{\sigma}\colon \sigma\in\mathscr{E}_k\}$ are pairwise disjoint because of property (1).
Thus $A$ is a Cantor set. For any $d \geq 1$, pairwise distinct
$T_1',T_2',\dotsc ,T_d'\in \mathcal{T}\setminus\{\theta_{\mathcal{T}}\}$, and non-empty open subsets $U_1,U_2,\dotsc ,U_d$ and $V_1,V_2,\dotsc ,V_d$ of $X$ intersecting $A$,  by Proposition \ref{d-WM} it  is suffice to show that
\begin{align}\label{e000}
\bigcap_{j=1}^{d}\bigcap_{s\in\{1,2,\dotsc ,d\}^d}N(V_j\cap A;U_{s(1)},\dotsc ,U_{s(d)}\mid T_1,T_2,\dotsc ,T_d)\neq\emptyset.
\end{align}

We pick an integer $L$ large enough  such that $\{T_1',T_2',\dotsc ,T_d'\}\subset\{T_1,T_2,\dotsc ,T_L\}$ and there exists pairwise distinct $\sigma^1,\sigma^2,\dotsc ,\sigma^d$, ${\sigma^1}',{\sigma^2}',\dotsc ,{\sigma^d}'$ in $\mathscr{E}_L$, such that $A_{\sigma^i}\subseteq U_i$ and $A_{{\sigma^i}'}\subseteq V_i$, for $i=1,2,\dotsc ,d$. Without loss of generality, we can assume $T_i'=T_i$ for $i=1,2,\dotsc ,d$.
Then there exists $n_{L+1}\in\N$  such that
$$A_{\sigma}\subset A_{\sigma(1)}\cap T_1^{-n_{L+1}}A_{\sigma(2)}\cap\dotsc \cap T_{L}^{-n_{L+1}}A_{\sigma(L+1)}\neq\emptyset,$$
for any $\sigma=\{\sigma(1),\sigma(2),\dotsc ,\sigma(L+1)\}\in\mathscr{E}_{L+1}$. In particularly, for any $j=1,2,\dotsc d$ and $s\in\{1,2,\dotsc ,d\}^d$, let
$\sigma_{js}=\{{\sigma^j}',\sigma^{s(1)},\dotsc ,\sigma^{s(d)},\eta^1,\dotsc ,\eta^{L-d}\}\in\mathscr{E}_{L+1}$, where $\eta^1,\eta^2,\dotsc ,\eta^{L-d}$ is any $L-d$ elements of $\mathscr{E}_{L}$. Then
$$A_{\sigma_{js}}\subset A_{{\sigma^j}'}\cap T_1^{-n_{L+1}}A_{\sigma^{s(1)}}\cap\dotsc \cap T_{d}^{-n_{L+1}}A_{\sigma^{s(d)}}.$$
Since $A_{\sigma_{js}}\cap A\neq \emptyset$, there exists
$$v_{js}\in A\cap A_{{\sigma^j}'}\cap T_1^{-n_{L+1}}A_{\sigma^{s(1)}}\cap\dotsc \cap T_{d}^{-n_{L+1}}A_{\sigma^{s(d)}}.$$
Thus $v_{js}\in A\cap V_j\cap T_1^{-n_{L+1}}U_{s(1)}\cap\dotsc \cap T_d^{-n_{L+1}}U_{s(d)}$, for any $j=1,2,\dotsc ,d$ and $s\in\{1,2,\dotsc ,d\}^d$.
Thus (\ref{e000}) holds, which ends the proof.
\end{proof}

\begin{rem}\label{solvable}
Furstenberg introduced the following example (see \cite[P.40]{F}). Let $X=\{-1,1\}^\Z$, $T$ be the shift map: $T\omega(n)=\omega(n+1)$
and  $R:X\rightarrow X$ be defined by:
\[ R\omega(n)=\begin{cases}
\omega(n),& for\ n=0 \\
-\omega(n),& for\ n\neq 0.
\end{cases}\]
It is clear that $R^2=id_X$.
Let $S=RTR$. Then $S^n=RT^nR$  for any $n\in\N$.
The group $G$ generated by $T$ and $R$ is a solvable group.
For any $\omega\in X$, let
$U_{\omega}=\{x\in X \colon  x(0)=\omega(0)\}$.
Then $T^n\omega\in U_{\omega}$ if and only if $\omega(n)=\omega(0)$, and $S^n\omega\in U_{\omega}$ if and only if $\omega(n)=-\omega(0)$. Hence
$(\omega,\omega)\notin\overline{orb_+((\omega,\omega),T\times S)}$. Thus for any closed subset $E$ of $X$ with $|E|\geq 2$, any $\omega\in E$, we have
$$\overline{orb_+((\omega,\omega), T\times S)}\nsupseteq E^2.$$
This implies that there is no $\Delta$-transitive subsets in $(X,G)$. For this group $G$, there exists a finitely generated torsion-free discrete solvable group $F$ with the unit $e_F$ such that there is a surjective homomorphism $\pi: F\to G$.
Let $(X,F)$ be the $F$-system, where the group actions is defined as
$$f(\omega)=\pi(f)(\omega), \text{ for } f\in F \text{ and }\omega\in X. $$
Taking $T_F$ and $S_F\in F$ such that $\pi(T_F)=T$ and $\pi(S_F)=S$, one has
$(\omega,\omega)\notin\overline{orb_+((\omega,\omega),T_F\times S_F)}$ for any $\omega\in X$.
Thus there is no $\Delta$-transitive subset in $(X,F)$.
\end{rem}

From Remark \ref{solvable}, we can obtain the following proposition.

\begin{prop}\label{solvable2} There exists a $F$-system $(Z,F)$, where $F$ is a finitely generated torsion-free discrete solvable group  such that $h_{top}(Z,F)>0$ but there are no $\Delta$-transitive subsets in $(Z,F)$.
\end{prop}
\begin{proof} Let $(X,F)$ be the $F$-system as in the Remark \ref{solvable},  and $Y=\{0,1\}^F$.
$F$ acts on $Y$ as $(gy)(h)=y(hg^{-1})$ for any $g,h\in F$ and $y\in Y$.
Now we consider the product $F$-system $(X\times Y,F)$.
We will show that $h_{top}(X\times Y, F)>0$ and there is no $\Delta$-transitive subsets in $(X\times Y,F)$.

Let $\mathcal{U}_0=\{[0],[1]\}$, where $[i]=\{y\in Y\ | y(e_F)=i \}$ for $i=0,1$. Then
$g^{-1}\mathcal{U}_0=\{[0]_g, [1]_g\}$, where $[i]_g=\{y\in Y\ |\ y(g)=i \}$ for $i=0,1$.
Let $\{F_n\}_{n=1}^{\infty}$ be a F\o{}lner sequence of $F$. Then
$$h_{top}(F,\mathcal{U}_0)=\lim_{n\to\infty}\frac{H({\mathcal{U}_{F_n}})}{|F_n|}=\lim_{n\to\infty}\frac{\log2^{|F_n|}}{|F_n|}=\log 2.$$
Thus $h_{top}(X\times Y,F)\geq h_{top}(Y,F)\geq \log2>0$.

If there exists a $\Delta$-transitive subset $E$ of $(X\times Y, F)$, then there exists $(\omega,y)\in E$ such that
$$\overline{orb_+(((\omega,y),(\omega,y)),T_F\times S_F)}\supseteq E^2. $$
In particularly, $(\omega,\omega)\in \overline{orb_+((\omega,\omega),T_F\times S_F)}$, which contradicts the Remark \ref{solvable}.
Thus $F$-system $(Z,F):=(X\times Y,F)$ has no $\Delta$-transitive subsets.
\end{proof}

\section{proof of theorem \ref{thm-B}}
In this section, firstly we present some properties of a $\Delta$-weakly mixing subset in a $\mathcal{T}$-system $(X,\mathcal{T})$, and then we will prove Theorem \ref{thm-B} and Corollary \ref{corollary}. These arguments in this section partially follows from the proof of the Theorem A in \cite{H-L-Y-Z}.

Let $(X,\mathcal{T})$ be a $\mathcal{T}$-system, where $\mathcal{T}$ is a countable torsion-free discrete group and $E$ is a closed subset of $X$.
For any $\varepsilon >0$,  $d\geq1$ and pairwise distinct $T_1,T_2,\dotsc T_d\in\mathcal{T}\setminus\{\theta_{\mathcal{T}}\}$, we say that a subset $A$ of $X$ is \emph{$(\varepsilon,T_1,T_2,\dotsc ,T_d)$-spread in $E$} if there are
$0<\delta<\varepsilon$, $m\in \mathbb{N}$ and pairwise distinct $z_1,z_2,\dotsc z_m\in X$ such that $A\subset \bigcup_{i=1}^mB(z_i,\delta)$ and for any maps $g_j:\{z_1,z_2,\dotsc ,z_m\}\to E $ for $j=1,2,\dotsc d$, there exists an integer $L>\frac{1}{\varepsilon}$ such that
$$T_j^LB(z_i,\delta)\subset B(g_j(z_i),\varepsilon)$$
for any $i=1,2,\dotsc ,m$ and $j=1,2,\dotsc ,d$.
Denote by $\mathcal{H}(\varepsilon,T_1,T_2,\dotsc ,T_d;E)$ the collection of all closed subsets of $X$ that are $(\varepsilon,T_1,T_2,\dotsc ,T_d)$-spread in $E$. Put
$$\mathcal{H}(E)=\bigcap_{d=1}^{+\infty}\bigcap_{\{T_1,T_2,\dotsc, T_d\}\subset\mathcal{T} \setminus \{\theta_{\mathcal{T}}\}} \bigcap_{k=1}^{+\infty}\mathcal{H}(\tfrac{1}{k},T_1,T_2,\dotsc ,T_d;E).$$
It is clear that $\mathcal{H}(E)$ is a  hereditary subset of $2^X$.
\begin{prop}\label{C_i} Let $(X,\mathcal{T})$ be a $\mathcal{T}$-system, where $\mathcal{T}$ is a countable torsion-free discrete group, and $E$ be a perfect subset of $X$.
If there exists an increasing sequence of subsets $C_1\subset C_2\subset\dotsc $ in $\mathcal{H}(E)\cap 2^E$ such that $C=\bigcup_{i=1}^{\infty}C_i$ is dense in $E$, then $E$ is a $\Delta$-weakly mixing subset of $X$.
\end{prop}
\begin{proof}
For any $d\geq1$, pairwise distinct $T_1,T_2,\dotsc ,T_d\in\mathcal{T}\setminus\{\theta_{\mathcal{T}}\}$, and non-empty open subsets $U_1,U_2,\dotsc ,U_d$, $V_1,V_2,\dotsc ,V_d$ of $X$ intersecting $E$, by
Proposition \ref{d-WM},
it is sufficient to show that
\begin{align}\label{e-3}
\bigcap_{i=1}^{d}\bigcap_{k=1}^{M}N(V_i\cap E;U_{s^k(1)},U_{s^k(2)},\dotsc ,U_{s^k(d)})\neq\emptyset,
\end{align}
where $M=|\{1,2,\dotsc ,d\}^{d}|$ and enumerate $\{1,2,\dotsc ,d\}^d$ as $\{s^1,s^2,\dotsc ,s^M\}$.

Since $U_i\cap E\neq\emptyset$, there exist $u_i\in U_i\cap E$ and $\varepsilon>0$ such that $B(u_i,\varepsilon)\subset U_i$ for   $i=1,2,\dotsc ,d$.
Since $V_i\cap E\neq\emptyset$ for $i=1,2,\dotsc ,d$, $E$ is perfect and $C$ is dense in $E$, we can pick
$v_{i1},v_{i2},\dotsc ,v_{iM}\in V_i\cap C$ for $i=1,2,\dotsc ,d$  such that $v_{ik}\neq v_{i'k'}$ whenever $(i,k)\neq(i',k')\in\{1,2,\dotsc ,d\}\times\{1,2,\dotsc ,M\}$.
Then we take an integer $K_0 $ large enough such that
$v_{ik}\in C_{K_0}$ for all $i=1,2,\dotsc ,d$ and $k=1,2,\dotsc ,M$.

Let $a_0=\min\{\rho(v_{ik},v_{{i'}{k'}}): (i,k)\neq (i',k'), 1\leq i,i'\leq d, 1\leq k,k'\leq M\}$.
Then $a_0>0$.
Take $0<\varepsilon_1<\min\{\frac{a_0}{2},\frac{\varepsilon}{2}\}$. Since $C_{K_0}\in\mathcal{H}(E)$,
there exist $0<\delta<\varepsilon_1$, $m\in\N$ and pairwise distinct $z_1,z_2,\dotsc ,z_m\in X$ such that $C_{K_0}\subset\bigcup_{i=1}^{m}B(z_i,\delta)$ and for any maps $g_j:\{z_1,z_2,\dotsc ,z_m\}\to E$, $j=1,2,\dotsc ,d$, there exists an integer $L>\frac{1}{\varepsilon_1}$ such that $T_j^LB(z_i,\delta)\subseteq B(g_j(z_i),\varepsilon_1)$ for any $i=1,2,\dotsc ,m$ and $j=1,2,\dotsc ,d$.

Now we can pick $n_{ik}\in\{1,2,\dotsc ,m\}$ for $i=1,2,\dotsc ,d$ and $k=1,2,\dotsc ,M$, such that $v_{ik}\in B(z_{n_{ik}},\delta)$. Since $\delta<\frac{a_0}{2}$, one has $z_{n_{ik}}\neq z_{n_{i'k'}}$ whenever $(i,k)\neq(i',k')\in\{1,2,\dotsc ,d\}\times\{1,2,\dotsc ,M\}$.
Thus $Md\leq m$. For $j=1,2,\dotsc ,d$,  we define $h_j:\{z_1,z_2,\dotsc ,z_m\}\to E$  as
\[h_j{(z_p)}=\begin{cases}
u_{s^k(j)},&\text{ if there exists }z_{n_{ik}}\text{ such that } z_p=z_{n_{ik}},\\
u_{d},& \text{others}.
\end{cases}\]
Then there exists $L^{\ast}\in\N$ such that $T_j^{L^{\ast}}B(z_p,\delta)\subseteq B(h_j(z_p),\varepsilon_1)$ for any $j=1,2,\dotsc ,d$ and $p=1,2,\dotsc ,m$. In particularly, for any $i=1,2,\dotsc ,d$ and $k=1,2,\dotsc ,M$,
$$T_j^{L^{\ast}}B(z_{n_{ik}},\delta)\subset B(u_{s^k(j)},\varepsilon_1)\subset U_{s^k(j)},$$
and so $v_{ik}\in T_j^{-L^{\ast}}U_{s^k(j)}$. Thus for any $1\leq k\leq M$ and $1\leq i\leq d$
$$L^{\ast}\in N(V_i\cap E; U_{s^k(1)},U_{s^k(2)},\dotsc ,U_{s^k(d)} \mid T_1,T_2,\dotsc ,T_d)\neq\emptyset,$$
that is (\ref{e-3}) holds. This ends the proof.
\end{proof}

\begin{thm}\label{H(E)}Let $(X,\mathcal{T})$ be a $\mathcal{T}$-system, where $\mathcal{T}$ is a countable torsion-free discrete group. If $E$ is a closed subset of $X$ with $|E|\geq2$, then $E$ is $\Delta$-weakly mixing if and only if $E$ is perfect and
$2^E\cap \mathcal{H}(E)$ is a residue subset of $2^E$.
\end{thm}
\begin{proof}
Sufficiency. If $E$ is perfect and $ 2^E\bigcap \mathcal {H}(E)$ is a residue subset of $2^E$, then we can immediately obtain that $E$ is a $\Delta$-weakly mixing subset of $(X,\mathcal{T})$  by Lemma \ref{K_M} and Proposition \ref{C_i}, since $ 2^E\bigcap \mathcal {H}(E)$ is also a hereditary subset of $2^E$.

Necessity. Suppose $E$ is a $\Delta$-weakly mixing subset of $(X,\mathcal{T})$. By Remark \ref{remark 1}, $E$ is perfect.
To show that
$\mathcal{H}(E)\cap2^E$ is a residue subset of $2^E$,
it is suffice to show that for any given $\varepsilon>0$, $d\geq 1$ and pairwise distinct $T_1,T_2,\dotsc ,T_d \in \mathcal{T} \setminus\{\theta_{\mathcal{T}}\}$, one has
$\mathcal{H}(\varepsilon, T_1,T_2,\dotsc ,T_d;E)\cap 2^E$ is a dense open subset of $2^E$.

For any $A \in \mathcal{H}(\varepsilon, T_1,T_2,\dotsc ,T_d;E)$, by the definition there exist $\delta\in(0,\varepsilon)$, $m\in\N$, and pairwise distinct points $z_1,z_2,\dotsc ,z_m \in X$ such that
$$A\in\langle B(z_1,\delta),B(z_2,\delta),\dotsc,B(z_m,\delta)\rangle\subseteq\mathcal{H}(\varepsilon,T_1,T_2,\dotsc ,T_d;E).$$
Thus $\mathcal{H}(\varepsilon,T_1,T_2,\dotsc ,T_d;E)$$\cap 2^E$ is an open subset of $2^E$.

Now we shall show that for any fixed $n\in\N$ and non-empty open subsets $U_1,$ $U_2,\dotsc ,U_n$ of $X$ intersecting $E$,
\begin{align}\label{e-7}
\langle U_1,U_2,\dotsc ,U_n\rangle\cap\mathcal{H}(\varepsilon,T_1,T_2,\dotsc ,T_d;E)\cap 2^E\neq\emptyset.
\end{align}
This implies $\mathcal{H}(E)\cap 2^E$ is dense in $2^E$.

Since $E$ is a $\Delta$-weakly mixing subset of $(X,\mathcal{T})$ and $U_i\cap E\neq\emptyset$ for $i=1,2,\dotsc ,n$,
there exists $u_i\in U_i\cap E$ for $i=1,2,\dotsc ,n$ such that the orbit closure of $d$-tuple $(u,u,\dotsc ,u)$ under the action $T_1\times T_2\times\dotsc \times T_d$ contains $E^n\times E^n\times\dotsc \times E^n$ ($d$-times), that is
$$\overline{orb_{+}((u,u,\dotsc ,u),T_1\times T_2\times\dotsc \times T_d)}\supseteq E^n\times E^n\times\dotsc \times E^n\ (d\text{-times}),$$
where $u=(u_1,u_2,\dotsc ,u_n)$.
Since $X$ is compact and $E$ is a closed subset of $X$, there exists $m_0\in \mathbb{N}$ and $z_1,z_2,\dotsc ,z_{m_0} \in X$ such that
$E\subset \bigcup_{i=1}^{m_0}B(z_i,\frac{1}{2}\varepsilon)$ and $B(z_i,\frac{1}{2}\varepsilon)\cap E\neq \emptyset$ for any $1\leq i\leq m_0$.
We can arrange the $d$-tuple on the set $\{ 1,2,\dotsc ,m_0\}^n$ as the finite sequence $ \{\alpha^1,\alpha^2,\dotsc ,\alpha^L \}$, where $\alpha^k=(\alpha_1^k,\alpha_2^k,\dotsc ,\alpha_d^k)$ and $\alpha_j^k\in\{1,2,\dotsc ,m_0\}^n$ for $k=1,2,\dotsc ,L$ and $j=1,2,\dotsc ,d$.
For $\alpha^1$, there exists $n_1 \in \mathbb{N}$ such that $T_j^{n_1}(u_i)\in B(z_{\alpha_j^1(i)},\frac{1}{2}\varepsilon)$ for $i=1,2,\dotsc ,n$ and $j=1,2,\dotsc ,d$. Moreover, since $T_1,T_2,\dotsc ,T_d$ are continuous, we can find a neighborhood $W_i^1$ of $u_i$ such that $W_i^1\subset U_i$ and
$$T_j^{n_1}(W_i^1)\subset B(z_{\alpha_{j}^1(i)},\tfrac{1}{2}\varepsilon),$$
for any $i=1,2,\dotsc ,n$ and $j=1,2,\dotsc ,d$.
Then replacing $U_i$ by $W_i^1$ for $i=1,2,\dotsc ,n$, we can obtain $n_2\in\N$ and a neighborhood $W_i^2$ of $u_i$ such that $W_i^2\subset W_i^1$ and $T_j^{n_2}(W_i^2)\subset B(z_{\alpha _j^2(i)},\frac{1}{2}\varepsilon)$  for $i=1,2,\dotsc ,n$ and $j=1,2,\dotsc ,d$. We continue inductively obtaining positive integers $n_3,n_4\dotsc ,n_L$ and non-empty open subsets $W_i^k$ of $X$ intersecting $E$ such that
\begin{align}\label{e-4}
W_i^L\subset W_i^{L-1}\subset\dotsc \subset W_i^1 \subset U_i,\ T_j^{n_k}(W_i^k)\subset B(z_{\alpha_j^k(i)},\frac{\varepsilon}{2}),
\end{align}
for $i=1,2,\dotsc ,n$, $j=1,2,\dotsc ,d$ and $k=1,2,\dotsc ,L$.

Now we take $\omega_i\in W_i^L\cap E$, and $0<\delta<\frac{\varepsilon}{2}$ such that $B(\omega_i,2\delta)\subset W_i^L$ for any $i=1,2,\dotsc ,n$.
Let $W=\bigcup_{i=1}^{n}\overline{B(\omega_i,\delta)}\subset \bigcup_{i=1}^nB(\omega_i,2\delta)$.
 Then $W\in\langle U_1,U_2,\dotsc ,U_n\rangle$.
For any maps
$g_j:\{\omega_1,\omega_2,\dotsc ,\omega_n\}\to E$ with $j=1,2,\dotsc ,d$, there exists $1\leq h \leq L$ such that
\begin{align}\label{e-5}
g_j(\omega_i)\in B(z_{\alpha^h_j(i)},\tfrac{1}{2}\varepsilon),
\end{align}
for any $i=1,2,\dotsc ,n$ and $j=1,2,\dotsc ,d$.
Combining (\ref{e-4}) and (\ref{e-5}), one has
$$ T_j^{n_h}B(\omega_i,2\delta)\subset B(z_{\alpha^h_j(i)},\tfrac{1}{2}\varepsilon)\subset B(g_j(\omega_i),\varepsilon),$$
for $i=1,2,\dotsc ,n$ and $j=1,2,\dotsc ,d$.
Thus $W\cap E\in \mathcal{H}(\varepsilon,T_1,T_2,\dotsc ,T_d;E)\cap E$, and (\ref{e-7}) holds. This finishes the proof.
\end{proof}

Now let us prove Theorem \ref{thm-B}.
\begin{proof}[Proof of Theorem \ref{thm-B}]Let $E$ be a closed subset of $X$ with $|E|\geq 2$. Suppose $E$ is a $\Delta$-weakly mixing subset of $(X,\mathcal{T})$.
Then $E$ is perfect and $\mathcal{H}(E)\cap 2^E$ is a residue subset of $2^E$ by Theorem \ref{H(E)}.
Since $\mathcal{H}(E)\cap 2^E$ is also a hereditary subset of $2^E$, there exists an increasing sequence of Cantor sets $C_1\subset C_2\subset\dotsc \subset E$ such that $C_i\in \mathcal{H}(E)\cap 2^E$ and $C=\bigcup_{i=1}^{\infty}C_i$ is dense in $E$ by Lemma \ref{K_M}.

Let $A$ be a subset of $C$, $d\geq 1$, pairwise distinct $T_1,T_2,\dotsc T_d\in\mathcal{T}\setminus\{\theta_{\mathcal{T}}\}$, and $g_j:A \to E$ be continuous maps for $j=1,2,\dotsc,d$. For any $k\in\N$, let $A_k=C_k\cap A$. Then the closure $\overline{A_k}$ of $A_k$ is also in $\mathcal{H}(E)$, since $\mathcal{H}(E)$ is hereditary.

By the definition of $\mathcal{H}(E)$, there exists $0<\delta_k<\frac{1}{k}$, $m_k\in\N$, and $z_1^k,z_2^k,\dotsc,z_{m_k}^k\in X$, such that
$$\overline{A_k}\subset\bigcup_{i=1}^{m_k}B(z_i^k,\delta_k)$$
with $A_k\cap B(z_i^k,\delta_k)\neq\emptyset$, $i=1,2,\dotsc,m_k$, and for any maps $h_j:\{z_1^k,z_2^k,\dotsc ,z_{m_k}^k\}\to E$, $j=1,2,\dotsc,d$, there exists $L>k$ such that $T_j^LB(z_i^k,\delta_k)\subseteq B(h_j(z_i^k),\frac{1}{k})$, for any $i=1,2,\dotsc,m_k$ and $j=1,2,\dotsc,d$.

For any $i=1,2,\dotsc,m_k$, there exists $u^k_i\in A_k\cap B(z_i^k,\delta_k)$. Now we define $\widetilde{g_j}:\{z_1^k,z_2^k,\dotsc,z_{m_k}^k\}\to E $ as $\widetilde{g_j}(z_i^k)=g_j(u_i^k)$, for any $j=1,2,\dotsc,d$ and $i=1,2,\dotsc,m_k$. Then there exists $q_k>k$ such that $T_j^{q_k}B(z_i^k,\delta_k)\subset B(\widetilde{g}(z_i^k),\frac{1}{k})$, for any $j=1,2,\dotsc,d$ and $i=1,2,\dotsc,m_k$. We show that the sequence $\{q_k\}$ is as required.

For any $x\in A$, there exists $K_0\in\N$ such that for any $k>K_0$ $x\in A_k$. For any $k>K_0$, there exists $z_{i_x}^k\in\{z_1^k,z_2^k,\dotsc,z_{m_k}^k\}$ such that $x\in B(z_{i_x}^k,\delta_k)$. Then for any $j=1,2,\dotsc,d$ we have
\begin{align*}
\lim_{k\to+\infty}\rho(T_j^{q_k}x,g_j(x))&\leq\lim_{k\to+\infty} \big(\rho(T_j^{q_k}x,\widetilde{g}_{j}(z_{i_x}^k))+\rho(\widetilde{g}_{j}(z_{i_x}^k),g_j(x))\big)\\
&\leq \lim_{k\to+\infty}\frac{1}{k}+\lim_{k\to+\infty}\rho(g_{j}(u_{i_x}^k),g_j(x))\\
&=\lim_{k\to+\infty}\rho(g_{j}(u_{i_x}^k),g_j(x)).
\end{align*}
Since $\rho(x,u_{i_x}^k)<\frac{2}{k}$ for any $k>K_0$, and $g_j$ is continuous for $j=1,2,\dotsc,d$, $$\lim_{k\to+\infty}\rho(g_{j}(u_{i_x}^k),g_j(x))=0.$$ 
Thus $\lim_{k\to+\infty}T_j^{q_k}x=g_j(x)$ for any $x\in A_k$, which ends the proof of necessity.

Sufficiency. For any $d\geq 1$, pairwise distinct $T_1,T_2,\dotsc ,T_d\in\mathcal{T}\setminus\{\theta_{\mathcal{T}}\}$, non-empty open subsets $U_1,U_2,\dotsc ,U_d$ and $V_1,V_2,\dotsc ,V_d$ of $X$ intersecting $E$, by Proposition \ref{d-WM} we need to show that
\begin{align}\label{e-6}
\bigcap_{s\in\{1,2,\dotsc ,d\}^{d+1}}N(V_{s(1)}\cap E;U_{s(2)},\dotsc ,U_{s(d+1)})\neq\emptyset.
\end{align}

Let $M=|\{1,2,\dotsc ,d\}^d|$ and enumerate $\{1,2,\dotsc ,d\}^d$ as $\{s^1,s^2,\dotsc ,s^M\}$.
Since $U_i\cap E\neq\emptyset$, there exists $u_i\in U_i\cap E$ and $\varepsilon>0$ such that $B(u_i,\varepsilon)\subset U_i$ for any $i=1,2,\dotsc ,d$.
Since $V_i\cap E\neq\emptyset$ for $i=1,2,\dotsc ,d$, $E$ is perfect, and $C$ is dense in $E$,
we can pick $v_{i1},v_{i2},\dotsc ,v_{iM}\in (V_i\cap E)\cap C$ for $i=1,2,\dotsc ,d$ such that $v_{il}\neq v_{i'l'}$ whenever $(i,l)\neq (i',l')\in\{1,2,\dotsc ,d\}\times\{1,2,\dotsc ,M\}$.

Let $A:=\{v_{il}: 1\leq i\leq d, 1\leq l\leq M\}$. Then $A$ is a subset of $C$. We define $g_j: A\to E$ as $g_j(v_{il})=u_{s^l(j)}$, for any $i,j=1,2,\dotsc ,d$ and $l=1,2,\dotsc ,M$. Then there exists an increasing sequence $\{q_k\}_{k=1}^{+\infty}$ of positive integers
such that
$$\lim_{k\to\infty}T_j^{q_k}v_{il}=g_j(v_{il})=u_{s^l(j)}\in U_{s^l(j)},$$
for any $i,j=1,2,\dotsc ,d$ and $l=1,2,\dotsc ,M$.
Thus we can pick $k_0\in\N$ large enough such that $v_{il}\in T_j^{-q_{k_0}}U_{s^l(j)}$ for any $1\leq i,j\leq d$ and $1\leq l\leq M$ that is (\ref{e-6}) holds. This ends the proof.
\end{proof}
Now we will prove Corollary \ref{corollary}.
\begin{proof}[Proof of Corollary \ref{corollary}] Since $h_{top}(X,\mathcal{T})>0$, there exists a $\Delta$-weakly mixing subset $E$ of $(X,\mathcal{T})$ by Theorem \ref{thm-A}. Then $E$ is perfect by Remark \ref{remark 1}, and by Theorem \ref{thm-B} there exists increasing sequence of Cantor subsets $C_1\subset C_2\subset\dotsc $ of $E$  such that $C=\bigcup_{i=1}^{\infty}C_i$ is dense in $E$  and satisfies the property in Theorem \ref{thm-B}. Since $|E|\geq 2$, we can pick distinct $e_1,e_2\in E$, and let
$\delta =\rho(e_1,e_2)>0$.

Given two distinct points $x,y\in C$  and $T_1,T_2\in\mathcal{T}\setminus\{\theta_{\mathcal{T}}\}$, there are two cases.

Case 1: $T_1=T_2=T$ for some $T\in\mathcal{T}\setminus\{\theta_{\mathcal{T}}\}$. Let $g:\{x,y\}\to E$ with $g(x)=g(y)=e_1$, and $g':\{x,y\}\to E$ with $g'(x)=e_1$, $g'(y)=e_2$. Then by the property in Theorem \ref{thm-B}, there exist two increasing sequences $\{p_k\}_{k=1}^{+\infty}$ and $\{p_k'\}_{k=1}^{+\infty}$ of positive integers  such that
$$\lim_{k\to\infty}\rho(T^{p_k}x,T^{p_k}y)=\rho(g(x),g(y))=0,$$ $$\lim_{k\to\infty}\rho(T^{p_k'}x,T^{p_k'}y)=\rho(g'(x),g'(y))=\delta.$$

Case 2: $T_1\neq T_2$. Let $g_1:\{x,y\}\to E$ with $g_1(x)=g_1(y)=e_1$, and $g_2:\{x,y\}\to E$ with $g_2(x)=g_2(y)=e_1$. Then by the property in Theorem \ref{thm-B}, there exists an increasing sequence $\{q_k\}_{k=1}^{+\infty}$ of positive integers such that
$$\lim_{k\to\infty}\rho(T_1^{q_k}x,T_2^{q_k}y)=\rho(g_1(x),g_2(y))=0. $$
Next let $g_1':\{x,y\}\to E$ with $g_1'(x)=g_1'(y)=e_1$, and $g_2':\{x,y\}\to E$ with $g_2'(x)=g_2'(y)=e_2$. Then by the property in Theorem \ref{thm-B}, there exists an increasing sequence $\{q_k'\}_{k=1}^{+\infty}$ of positive integers such that
$$\lim_{k\to\infty}\rho(T_1^{q_k'}x,T_2^{q_k'}y)=\rho(g_1'(x),g_2'(y))=\delta.$$

Summing up, one has $\liminf_{n\to+\infty}\rho(T_1^nx,T_2^ny)=0$ and $\limsup_{n\to+\infty}\rho(T_1^nx,T_2^y)$ $\geq\delta$ for any $x\neq y\in C$ and $T_1,T_2\in\mathcal{T}\setminus\{\theta_{\mathcal{T}}\}$. Thus $(X,\mathcal{T})$ is asynchronous chaotic.
\end{proof}

\bibliographystyle{amsplain}

\begin{thebibliography}{10}
\bibitem{Akin}
E. Akin, Lectures on Cantor and Mycielski sets for dynamical systems, Chapel Hill Ergodic Theory Workshops, 21-79. Contemp. Math., 356. Amer. Math. Soc., Providence, RI, 2004.

\bibitem{B_H}
F. Blanchard and W. Huang, Entropy sets, weakly mixing sets and entropy capacity, Discrete Contin. Dyn. Syst. 20(2008), no.2, 275--311.

\bibitem{B-G-K-M}
F. Blanchard, E. Glasner, S. Kolyada, A. Maass, On Li¨CYorke pairs, J. Reine Angew. Math. 547 (2002) 51¨C6

\bibitem{Chen-Li-Lv}
J. Chen, J. Li and J. L\"u, On multi-transitivity with respect to a vector, Sci. China Math. 57(2014), no.8, 1639--1648.

\bibitem{Chen-Li-Lv2}
Z. Chen, J. Li and J. L\"u, Point transitivity, $\Delta$- transitivity and multi-minimality, Ergodic Theory Dynam. Systems. 35(2015), no.5, 1423--1442.

\bibitem{Connes-Feldman-Weiss}
A. Connes, J. Feldman and B. Weiss, An amenable equivalence relation is generated by a single transformation, Ergodic Theory Dynam. Systems. 1(1981), no.4, 431--450.

\bibitem{Danilenko}
A.I. Danilenko, Entropy theory from the orbital point of view, Monatsh. Math. 134(2001), no.2, 121--141.

\bibitem{Dinaburg}
E.I. Dinaburg, A correlation between topological entropy and metric entropy, Dokl. Akad. Nauk. SSSR 190(1970), 19--22.

\bibitem{Emerson}
W.R. Emerson, The pointwise ergodic theorem for amenable group, Amer. J. Math. 96(1974), 472--487.

\bibitem{F}
H. Furstenberg, Recurrence in ergodic theory and combinatorial number theory, Princeton Univ. Press, Princeton, NJ., 1981.

\bibitem{FW}
H. Furstenberg and B. Weiss, Topological dynamics and combinatorial number theorey, J. Anal. Math. 34(1978), 61--85.

\bibitem{Glasner}
E. Glasner, Topological ergodic decompositions and applications to products of powers of a minimal transformation, J. Anal. Math. 64(1994), 241--262.


\bibitem{Glasner-Thouvenot-Weiss}
E. Glasner, J.P. Thouvenot and B. Weiss, Entropy theory without a past, Ergodic Theory Dynam. Systems 20(2000), no.5, 1355--1370.

\bibitem{Goodman}
T.N.T. Goodman, Relating topological entropy and measure entropy, Bull. London Math. Soc. 3(1971), no.3, 176--180.

\bibitem{Goodwyn}
L.W. Goodwyn, topological entropy bounds measures-theoretic entropy, Proc. Amer. Math. Soc. 23(1969), 679--688.

\bibitem{H-L-Y-Z}
W. Huang, J. Li, X. Ye and X. Zhou, Positive topological entropy and $\Delta$-weakly mixing sets, Adv. Math. 306(2017), 653--683.

\bibitem{Huang-Shao-Ye}
W. Huang, S. Shao and X. Ye, Topological correspondence of multiple ergodic averages of nilpotent group actions, J. Anal. Math., to appear, arxiv: 1604.07113.

\bibitem{Huang-Xu-Yi}
W. Huang, L. Xu and Y. Yi, Asymptotic pairs, stable sets and chaos in positive entropy systems, J. Funct. Anal. 268(2015), no.4, 824--846.

\bibitem{H-Y-Z}
W. Huang, X. Ye and G. Zhang, Local entropy theory for a countable discrete amenable group action, J. Funct. Anal. 261(2011), no.4, 1028--1082.

\bibitem{Kammeyer-Rudolph}
J. Kammeyer and D.J. Rudolph, Restricted orbit equivalence for actions of discrete amenable groups, Cambridge Tracts in Mathematics, 146. Cambridge University Press, Cambridge, 2002.

\bibitem{Kerr-Li2}
D. Kerr and H. Li, Combinatorial independence and sofic entropy. Commun. Math. Stat. 1(2013), no.2, 213¨C257. 37E05.

\bibitem{Kerr-Li}
D. Kerr and H. Li, Independence in topological and C*-dynamics. Math. Ann. 338(2007), no.4, 869--926.

\bibitem{Kieffer}
J. Kieffer, A generalized Shannon-McMillan theorem for the action of an amenable group on a probability space, Ann. Probab. 3(1975), no.6, 1031--1037.

\bibitem{Kwietniak-Li-Oprocah-Ye}
D. Kwietniak, J. Li, P. Oprocah and X. Ye, Multi-recurrence and van der Waerden systems, Sci. China Math. 60(2017), no.1, 59--82.

\bibitem{Kwietniak-Oprocah}
D. Kwietniak and P. Oprocah, On weak mixing, minimality and weak disjointness of all iterates, Ergodic Theorey Dynam. Systems. 32(2012), no.5, 1661--1672.

\bibitem{Leibman0}
A. Leibman, Multiple recurrence theorem for nilpotent group actions, Geom. Funct. Anal. 4(1994), no.6, 648--659.

\bibitem{Leibman}
A. Leibman, Multiple recurrence theorem for measure preserving actions of a nilpotent group, Geom. Funct. Anal. 8(1998), no.5, 853--931.

\bibitem{Li-Yorke}
T.Y. Li and J.A. Yorke, Period three implies chaos, Amer. Math. Month. 82(1975), 985-992.

\bibitem{Lindenstrauss}
E. Lindenstrauss, Pointwise theorems for amenable groups, Invent. Math. 146(2001), no.2, 259--295.

\bibitem{Lindenstrauss-Weiss}
E. Lindenstrauss and B. Weiss, Mean topological dimension, Israel J. Math. 115(2000), 1--24.

\bibitem{Moothathu}
T.K.S. Moothathu, Diagonal points having dense orbit, Colloq. Math. 120(2010), no.1, 127--138.

\bibitem{Ollagnier-Pinchon}
L.M. Ollagnier and D. Pinchon, The variational principle, Studia Math. 72(1982), no.2, 151--159.

\bibitem{Ornstein-Weiss}
D.S. Ornstein and B. Weiss. Entropy and isomorphism theorems for actions of amenable groups. J. Anal. Math. 48(1987), 1--141.

\bibitem{Parry}
W. Parry, Entropy and generators in ergodic theory, W. A. Benjamin, Inc., New York-Amsterdam, 1969.

\bibitem{Rudolph}
D.J. Rudolph and B. Weiss, Entropy and mixing for amenable group actions. Ann. of Math. (2) 151(2000), no.3, 1119--1150.

\bibitem{Stepin-Tagi-Zade}
A.M. Stepin and A.T. Tagi-Zade, Variational characterization of topological pressure of the amenable groups of transformations, Dokl. Akad. Nauk SSSR 254(1980), no.3, 545--549 (in Russian).

\bibitem{Ward-zhang}
T. Ward and Q, Zhang, The Abramov-Rokhlin entropy addition formula for amenable group actions, Monatsh. Math. 114(1992), no.3--4, 317--329.

\bibitem{Weiss}
B. Weiss, Actions of amenable group, Topics in Dynamics and Ergodic Theory, 226-262. London Math. Soc. Lecture Note Ser., 310,  Cambridge Univ. Press, Cambridge, 2003.



\end{thebibliography}

\end{document}